\documentclass{article}
\usepackage[left=1in, right=1in, top=1in, bottom=1in]{geometry}
\usepackage{enumitem}
\usepackage{authblk}
\usepackage{graphicx}	
\usepackage{amsmath, amsthm, amssymb,mathtools}
\usepackage{xspace}
\usepackage{comment}
\usepackage{hyperref}
\usepackage[dvipsnames]{xcolor}
\usepackage{algorithm}
\usepackage{algpseudocode}
\usepackage{caption}
\usepackage{makecell}
\usepackage{lineno}

\numberwithin{equation}{section}

\newcommand\omitstuff[1]{}

\newcommand{\rpart}{{\boldsymbol{\large{\lambda}}}}

\newcommand{\E}{\mathbb{E}}

\theoremstyle{plain}
\newtheorem{theorem}{Theorem}[section]
\newtheorem{proposition}[theorem]{Proposition}
\newtheorem{corollary}[theorem]{Corollary}
\newtheorem{lemma}[theorem]{Lemma}

\theoremstyle{definition}
\newtheorem{definition}[theorem]{Definition}
\newtheorem{remark}[theorem]{Remark}

\newtheorem{problem}[theorem]{Problem}
\newtheorem{ex}[theorem]{Example}

\newcommand{\wt}{\text{wt}}

\DeclareMathOperator{\inv}{inv}

\DeclareMathOperator{\des}{des}

\DeclareMathOperator{\maj}{maj}

\DeclareMathOperator{\pr}{Pr}

\DeclareMathOperator{\negative}{negsum}
\DeclareMathOperator{\Sym}{\mathfrak{S}} 
\DeclareMathOperator{\supp}{supp}

\title{Moments of Colored Permutation Statistics on Conjugacy Classes\footnote{This work was completed in part at the 2022 Graduate Research Workshop in Combinatorics, which was supported in part by NSF grant \#1953985 and a generous award from the Combinatorics Foundation. ML was partially supported by J. A. Grochow's NSF award CISE-2047756. MY was partially supported by the University of Denver's Professional Research Opportunities for Faculty Fund 80369-145601.}}

\author[1]{Jesse Campion Loth}
\author[2]{Michael Levet}
\author[3]{Kevin Liu}
\author[4]{Sheila~Sundaram}
\author[5]{Mei Yin} 

\affil[1]{Department of Mathematics, Simon Fraser University, Burnaby, BC, Canada}
\affil[2]{Department of Computer Science, College of Charleston, Charleston, SC, USA}
\affil[3]{Department of Mathematics and Computer Science, University of the South, Sewanee, TN, USA}
\affil[4]{School of Mathematics, University of Minnesota, Minneapolis, MN, USA}
\affil[5]{Department of Mathematics, University of Denver, Denver, CO, USA}

\begin{document}
\maketitle

\begin{abstract}
In this paper, we consider the moments of statistics on conjugacy classes of the colored permutation groups $\mathfrak{S}_{n,r}=\mathbb{Z}_r\wr \mathfrak{S}_n$. We first show that any fixed moment of a statistic coincides on all conjugacy classes where all cycles have sufficiently long length. Additionally, for permutation statistics that can be realized via a process we call order-invariant extension, these moments are polynomials in $n$. Finally, for the descent statistic on the hyperoctahedral group $B_n\cong \mathfrak{S}_{n,2}$, we show that its distribution on conjugacy classes without short cycles satisfies a central limit theorem. Our results build on and generalize previous work of Fulman (\textit{J. Comb. Theory Ser. A.}, 1998), Hamaker and Rhoades (arXiv, 2022), and Campion Loth, Levet, Liu, Stucky, Sundaram, and Yin (arXiv, 2023). In particular, our techniques utilize the latter combinatorial framework.
\end{abstract}

\bigskip 



\thispagestyle{empty}

\newpage

\setcounter{page}{1}


\section{Introduction}

For a finite group $G$, a (real) \textit{statistic} is a map $X : G \to \mathbb{R}$, and such a statistic induces a probability distribution given by $\pr[X=k]=|X^{-1}(k)|/|G|$. When $G = \mathfrak{S}_{n}$ is the symmetric group, $X$ is called a \textit{permutation statistic}. The study of permutation statistics is a classical topic in algebraic combinatorics; Stanley's texts \cite{StanEC1, StanEC2} serve as a key reference in this area. 

The refinement of permutation statistics to conjugacy classes $C_{\lambda}$  of $\mathfrak{S}_n$ has also been studied, e.g., see \cite{GesselReutenauer,Brenti1993,CooperJonesZhuang2020}. One notable result of Fulman \cite{FULMAN1999390} is that the $k$th moment of the descent statistic coincides on all conjugacy class of $\mathfrak{S}_n$ without cycles of length $1,2,\ldots,2k$, and in fact, on these conjugacy classes, it equals the corresponding moment on $\mathfrak{S}_n$. Since then, various generalizations of this result for other statistics have been established. Recently, Hamaker and Rhoades \cite{hamaker2022characters} introduced a notion of \textit{local} permutation statistics. Using representation-theoretic methods, they established that for a fixed statistic, each moment depends only on $n$ and the number of ``short" cycles in $\lambda$. In particular, this moment coincides on all conjugacy classes $C_{\lambda}$ when there are no ``short" cycles. Independently, Campion Loth, Levet, Liu, Stucky, Sundaram, and Yin \cite{GRWCPermutationStatistics} established similar results for these conjugacy classes in $\mathfrak{S}_n$ using only combinatorial techniques.

In this paper, we build on the framework in \cite{GRWCPermutationStatistics} to investigate the moments on conjugacy classes of \emph{colored permutation statistics}, i.e., statistics on the colored permutation group $\mathfrak{S}_{n,r}$, that is, the wreath product $\mathbb{Z}_r\wr \mathfrak{S}_n$. In contrast to the vast literature on permutation statistics in $\mathfrak{S}_{n}$, there has been considerably less work on statistics for the colored permutation groups $\mathfrak{S}_{n,r}$. The distributions of specific colored permutation statistics over the entire group $\mathfrak{S}_{n,r}$ have been studied by Steingr\'{i}msson \cite{steingrim94}, Fire \cite{fire2005statistics,FireFPSAC}, and Moustakas \cite{Moustakas2018TheED, MOUSTAKAS2021102177}, though we are not aware of any prior work on individual conjugacy classes except when $r=2$. In this case, the colored permutation groups $\mathfrak{S}_{n,2}$ are isomorphic to the hyperoctahedral groups $B_n$, which are the type $B$ Coxeter groups. Statistics on the entire group $B_n$ and its conjugacy classes have been studied extensively, e.g., by Reiner \cite{ReinerEuropJC1993-1, ReinerEuropJC1993-2,ReinerDiscMath1995}, Adin and Roichman \cite{AR2001,AdinRoichman1} with Gessel \cite{AdinGesselRoichman} and Brenti \cite{AdinBrentiRoichman}, and Brenti and Carnevale \cite{BrentiCarnevaleOddLength, BrentiCarnevaleOddLengthDiagrams}. Statistics on general Coxeter groups have also been studied, e.g., by Kahle and Stump \cite{Kahle2018}.

\noindent \\ \textbf{Main Results.} Since any fixed moment of a permutation statistic coincides on all conjugacy classes with no ``short" cycles, one might expect similar properties to hold for colored permutation statistics. We show that this is indeed the case. Similar to $\mathfrak{S}_n$, conjugacy classes of $\mathfrak{S}_{n,r}$ are determined by cycle type, with the primary difference in $\mathfrak{S}_{n,r}$ being that cycles now have a notion of color. Similar to the usage of $C_{\lambda}$ for conjugacy classes of $\mathfrak{S}_{n}$, we will use $C_{\rpart}$ to denote a conjugacy class of $\mathfrak{S}_{n,r}$, where $\rpart$  records cycle type.

For a colored permutation statistic $X:\mathfrak{S}_{n,r} \to \mathbb{R}$, let $\E_{\rpart}[X]$ denote its mean on the conjugacy class $C_{\rpart}$ of $\mathfrak{S}_{n,r}$. By generalizing the definitions in \cite{GRWCPermutationStatistics} and \cite{hamaker2022characters}, we define a notion of a colored permutation statistic having \emph{degree at most $m$}. This corresponds to the statistic being a linear combination of indicator functions for \emph{partial colored permutations} of size at most $m$. Using this, we establish the following result, which formalizes the statement that moments of colored permutation statistics stabilize on conjugacy classes without ``short" cycles. 

\begin{theorem}[cf. {\cite[Theorem~6.2]{hamaker2022characters}} and {\cite[Theorem~7.16]{GRWCPermutationStatistics}}] \label{thm:MainIndependence}
Suppose $X:\mathfrak{S}_{n,r}\to \mathbb{R}$ has degree at most $m$. For any $k\geq 1$, the $k$th moment $\E_{\rpart}[X^k]$ coincides on all conjugacy classes $C_{\rpart}$ of $\mathfrak{S}_{n,r}$ with no cycles of length $1,2,\ldots,mk$. 
\end{theorem}

Many statistics have natural decompositions that give bounds on their degrees. 
In the case of $\mathfrak{S}_n$, we refer the reader to \cite{GRWCPermutationStatistics} for numerous examples. In the case of the hyperoctahedral group $B_n\cong \mathfrak{S}_{n,2}$, the descent and inversion statistics have degree at most $2$, and we show in Theorem~\ref{thm:size2} that equality actually holds. However, we will also see in Theorems \ref{thm:mean-des} and \ref{thm:inversions} that their first moments coincide on all conjugacy classes with no cycles of length $1$, so the converse of Theorem~\ref{thm:MainIndependence} does not hold in general.

We also define an operation called \emph{order-invariant extension} on $\mathfrak{S}_{n,r}$ that generalizes one in \cite{GRWCPermutationStatistics} for $\mathfrak{S}_n$. This procedure starts with a sufficiently ``nice" statistic on some fixed $\mathfrak{S}_{n_0,r}$ and extends to a sequence of statistics $(X_n)_{n\geq 1}$ defined on $(\mathfrak{S}_{n,r})_{n\geq 1}$. For any $(X_n)_{n\geq 1}$ that can be constructed in this manner, we show that the $k$th moments on conjugacy classes without short cycles is given by a single polynomial in $n$.

\begin{theorem}[cf. {\cite[Theorem~1.4]{hamaker2022characters}} and {\cite[Theorem~7.26]{GRWCPermutationStatistics}}] \label{thm:Polynomiality}
Fix $r\ge 1$, and let $(X_{n})_{n\geq 1}$ be the $r$-colored order-invariant extension of an order-invariant set of partial colored permutations $\mathcal{C}$. Suppose $\mathcal{C}$ consists of partial colored permutations of size $m$. Then for any $k\geq 1$, there exists a polynomial $p(n)$ of degree at most $mk$ with the following property: for any $n\geq 1$ and conjugacy class $C_{\lambda_n}$ of $\mathfrak{S}_{n,r}$ with no cycles of length $1,2,\ldots,mk$, we have $p(n) = \E_{\rpart_{n}}[X_{n}^k]$.
\end{theorem}

Many statistics can be constructed using an order-invariant extension. In the case of $\mathfrak{S}_n$, several examples appear in \cite{GRWCPermutationStatistics}. In particular, permutation pattern statistics can be constructed as order-invariant extensions \cite[Example 7.23]{GRWCPermutationStatistics}. In $B_n$, the inversion statistic $\inv$ can be constructed using an order-invariant extension, which we show in Example~\ref{ex:inv_symmetric}.

Finally, in special cases, moments of a statistic on conjugacy classes without short cycles coincide with the moments on the entire group. We show this for the descent statistic on $B_n$, and using this fact, 
we establish a result on the asymptotic distribution of the descent statistic $\des_B$ on conjugacy classes of $B_n$ without short cycles. This is an analogue of a corresponding statement involving descents on conjugacy classes of $\mathfrak{S}_n$ established by Fulman in \cite{FulmanJCTA1998}. 

To state our result, we note that in $B_n$, the cycle type of a permutation is determined by a pair of partitions $(\lambda,\mu)$, so we will use $C_{\lambda,\mu}$ to denote a conjugacy class. Additionally, recall that a sequence of random variables $\{X_n\}_{n\geq 1}$ converges in distribution to $X$ if the cumulative distribution functions of $\{X_n\}_{n\geq 1}$ converge to the cumulative distribution function of $X$ at its points of continuity.

\begin{theorem} \label{thm:CLTMain}
For every $n\geq 1$, let $C_{\lambda_n, \mu_n}$ be a conjugacy class of $B_n$, and define $X_n$ to be the descent statistic $\des_B$ on $C_{\lambda_n,\mu_n}$. Suppose that for all $i$, the number of cycles of length $i$ in $\lambda_n$ and $\mu_n$ approaches 0 as $n\to\infty$. Then for sufficiently large $n$, $X_n$ has mean $n/2$ and variance $(n+1)/12$, and as $n\to\infty$, the random variable $(X_n-n/2)/\sqrt{(n+1)/12}$ converges in distribution to a standard normal distribution.
\end{theorem}

To prove Theorem~\ref{thm:CLTMain}, we leverage a generating function of Reiner \cite[Theorem 4.1]{ReinerEuropJC1993-2} for the joint distribution of descent and major index by cycle type (an analogue of the corresponding generating function for the symmetric group \cite{GesselReutenauer}), in tandem with a specialization of this result as observed in \cite[Theorem 5.3]{FulmanKimLeePetersen2021}. The arguments then follow Fulman \cite[Theorem 1 and proof of Theorem 2]{FulmanJCTA1998}, but the technical details are nontrivial and require care to execute.

\noindent \\ \textbf{Related Work.}
In the far-reaching paper \cite{hamaker2022characters} of Hamaker and Rhoades, the authors introduced \emph{$k$-local} statistics and studied their moments on conjugacy classes using expansion into the irreducible characters of $\mathfrak{S}_n$. We note here that their notion of \emph{$k$-local} coincides with the notion of \emph{degree at most $k$} in our present work. Hamaker and Rhoades also introduced a rich subclass of local permutation statistics called \textit{regular} statistics, and they showed that their means admit an expansion as a sum of rational functions in $n$ and polynomials in the number of short cycles in $\lambda$ and $n$. When all the cycles of $\lambda$ are sufficiently large, this expansion is in fact a polynomial in $n$ and the number of short cycles in the conjugacy class. In $\mathfrak{S}_n$, one can show that the order-invariant extensions in Theorem~\ref{thm:Polynomiality} are a subclass of these regular permutation statistics, so these results agree on conjugacy classes of $\mathfrak{S}_n$ without short cycles.

Prior to the work of Hamaker and Rhoades, Hultman \cite{Hultman} and Gill \cite{GillThesis} utilized the irreducible characters of $\mathfrak{S}_n$ to derive explicit formulas for the mean of several commonly-studied statistics on conjugacy classes of $\mathfrak{S}_n$. Subsequently, Gaetz and Ryba \cite{GaetzRyba} considered moments of permutation pattern statistics on conjugacy classes using the irreducible characters of $\mathfrak{S}_n$. They showed that for sufficiently large $n$, the $k$th moment of a permutation pattern statistic on conjugacy classes is a single polynomial in $n$ and the number of short cycles. Since permutation pattern statistics in $\mathfrak{S}_n$ can be constructed as order-invariant extensions, their result and our Theorem~\ref{thm:Polynomiality} agree on conjugacy classes of $\mathfrak{S}_n$ without short cycles.

Asymptotic results for several statistics on conjugacy classes of $\mathfrak{S}_n$ are known. Fulman \cite{FulmanJCTA1998} established asymptotic normality  on conjugacy classes without short cycles for descents and  major index. For descents, Kim and Lee \cite{KimLee2020} extended Fulman's result by showing asymptotic normality on arbitrary conjugacy classes of $\mathfrak{S}_n$, where the limit distribution depends only on the limiting proportion of fixed points. Our Theorem~\ref{thm:CLTMain} is an analogue of Fulman's original result for $B_n$, and we suspect an analogue of Kim and Lee's result holds on arbitrary conjugacy classes of $B_n$. Fulman, Kim, and Lee \cite{fulmankimlee} showed asymptotic normality for peaks in conjugacy classes of $\mathfrak{S}_n$, where the limit distribution again depends only on the limiting proportion of fixed points. Recent work of F\'{e}ray and Kammoun \cite{FerayKammoun} also established asymptotic normality of vincular patterns on conjugacy classes with appropriate conditions on the limiting proportion of fixed points and 2-cycles. 

\noindent \\ \textbf{Outline of Paper.} We start in Section~\ref{section:preliminaries} with preliminary information on the colored permutation and hyperoctahedral groups. In Section~\ref{section:colored}, we establish Theorems~\ref{thm:MainIndependence} and \ref{thm:Polynomiality}. We then give explicit formulas for the first moments of the descent and inversions statistics on conjugacy classes of the hyperoctahedral groups in Section~\ref{Bnstatistics}, and using these, we establish that the converse of Theorem~\ref{thm:MainIndependence} does not hold in general. In Section~\ref{section:CLT}, we further study the descent statistic on conjugacy classes of the hyperoctahedral group to establish Theorem \ref{thm:CLTMain}. We conclude in Section~\ref{section:conclusion} with open problems.

\section{Preliminaries}\label{section:preliminaries}

Throughout, we assume familiarity with the symmetric group $\mathfrak{S}_{n}$ as described in \cite{DummitFoote}. In particular, recall that the conjugacy classes of $\mathfrak{S}_n$ are indexed by integer partitions $\lambda$ of $n$. We can write $\lambda$ in two ways: as a decreasing tuple of positive integers $(\lambda_1\ge \lambda_2\ge\cdots)$ which sum to $n$, or, in multiplicative notation, as 
$\lambda=(1^{m_1(\lambda)}, 2^{m_2(\lambda)},\ldots),$ where $m_i(\lambda)$ denotes the number of parts of $\lambda$ that are equal to $i$. The corresponding conjugacy class is denoted $C_{\lambda}$. 

\subsection{Colored permutation groups}
We now give preliminary information about colored permutation groups. 

\begin{definition}
 The \emph{colored permutation group} $\mathfrak{S}_{n,r}$ is the wreath product 
    $\mathbb{Z}_r\wr \mathfrak{S}_n$,    
    where $\mathfrak{S}_n$ is the symmetric group on $n$ elements and $\mathbb{Z}_r$ is the cyclic group on $r$ elements. A \emph{colored permutation} is an element in $\mathfrak{S}_{n,r}$, and it can be expressed as $(\omega,\tau)$ where $\omega \in \mathfrak{S}_n$ and $\tau:[n]\rightarrow \mathbb{Z}_r$ is a function called a \emph{coloring}. The value $\tau(j)$ is called the \emph{color} of $j$.
\end{definition}

    The colored permutation group $\mathfrak{S}_{n,r}$ has a  canonical embedding as a subgroup of the symmetric group $\mathfrak{S}_{rn}$, which we describe explicitly as follows. 
    Let $[n]^r$ denote the set of $rn$ elements \[\{i^c \mid i\in [n], c\in \mathbb{Z}_r\},\] 
    where the exponent indicates the color of an element in $[n]$. We can view the colored permutation $(\omega, \tau)$ as a bijection $f:[n]^r \rightarrow [n]^r$ defined by  $f(i^c) = \omega(i)^{\tau(\omega(i))+c}$ for all $i\in [n]$ and $c\in \mathbb{Z}_r$.  An example is shown below.
 
\begin{ex}\label{ex:coloredpermutation}
Let $\omega\in \Sym_5$ be the permutation specified by \[\omega(1)=4,\, \omega(2)=5, \,\omega(3)=1,\, \omega(4)=3, \,\omega(5)=2.\] In one-line notation we write $\omega=[45132]$, and in cycle notation 
    $\omega=(143)(25)$. Now define $\tau:[5]\to \mathbb{Z}_4$ by
    \[\tau(1)=3,\, \tau(2) = 0,\, \tau(3)=1, \, \tau(4)=1,\, \tau(5)=3.\]
    Combined, this defines an element $(\omega,\tau)\in \mathfrak{S}_{5,4}$. Interpreting $(\omega,\tau)$ as a function 
    $f:[5]^4 \rightarrow [5]^4$ satisfying $f(i^0) = \omega(i)^{\tau(\omega(i))}$, the images of each $f(i^0)$ are as follows:
    $$
    f(1^0) = 4^1, f(2^0) = 5^3, f(3^0) = 1^3, f(4^0) = 3^1, f(5^0) = 2^0.
    $$
    Observe that this is enough to determine $f(i^c)$ for all $i^c\in [n]^r$. For example, $f(1^0) = 4^1$ also gives us that $f(1^1) = 4^2, f(1^2) = 4^3$, and $f(1^3) = 4^4 = 4^0$. Similar to conventions in $\mathfrak{S}_n$, one can express this colored permutation in two-line, one-line, and cycle notations as shown below:
    \begin{align*}
        (\omega,\tau)=\left(\begin{matrix}
1^0 & 2^0 & 3^0 & 4^0 & 5^0 \\
4^1 & 5^3 & 1^3 & 3^1 & 2^0 
\end{matrix}\right)=[4^1 5^3 1^3 3^1 2^0]=(1^3 4^1 3^1)(2^0 5^3).
    \end{align*}
    Note that the only values that appear in the one-line and cycles notations are the images of the elements colored $0$, i.e., the values $f(i^0)$ for $i\in [n]$. 
\end{ex}

We next describe the conjugacy classes of $\mathfrak{S}_{n,r}$. Similar to permutations in $\mathfrak{S}_n$, colored permutations have a notion of cycle type derived from the cycle notation.

\begin{definition}\label{def:cycles}
    For a cycle in a colored permutation $\mathfrak{S}_{n,r}$, its \emph{length} is the number of elements in it, and its \emph{color} is the sum of the colors that appear (as an element in $\mathbb{Z}_r$). The \emph{cycle type} of $(\omega,\tau)\in \mathfrak{S}_{n,r}$ is the multi-set of cycle lengths with colors obtained from its cycles. 
\end{definition}

We will use the following generalization of partitions to record cycle type. 

\begin{definition}
    An \emph{$r$-partition} of $n\in \mathbb{N}$ is an $r$-tuple of partitions $\rpart=(\lambda^c)_{c=0}^{r-1}$ where each $\lambda^c$ is a partition of some positive integer $n_c$ such that $\sum_{c=0}^{r-1} n_c=n$. When $r=2$, we also call this a \emph{bi-partition}. For any such $r$-partition $\rpart$, define $C_{\rpart}$ to denote the elements in $\mathfrak{S}_{n,r}$ with cycle type $\rpart$, i.e., those with $m_i(\lambda^c)$ cycles of length $i$ and color $c$ for all $i\in [n]$ and $c\in \mathbb{Z}_r$.
\end{definition}

\begin{ex}\label{ex:coloredcycletype}
    Consider the colored permutation in $\mathfrak{S}_{5,4}$ from Example~\ref{ex:coloredpermutation} expressed in cycle notation as $(1^34^13^1)(2^05^3)$. The first cycle has length $3$ and color $1$, and the second cycle has length $2$ and color $3$. The cycle type of this permutation is
    \[\rpart=(\lambda^{0},\lambda^{1},\lambda^{2},\lambda^{3})=(\emptyset ,(3),\emptyset ,(2)).\]
\end{ex}

The conjugacy classes of $\mathfrak{S}_{n,r}$ are well understood in terms of cycle type.

\begin{proposition}\cite[Theorem~4.2.8, Lemmas~4.2.9-4.2.10]{JamesKerber1981} \label{cycletype}
    Two colored permutations in $\mathfrak{S}_{n,r}$ are conjugate if and only if they share the same cycle type. Hence, the conjugacy classes of $\mathfrak{S}_{n,r}$ are given by $C_{\rpart}$, where $\rpart$ is an $r$-partition of $n$.  
\end{proposition}

Throughout this paper, we will use $\pr_{\mathfrak{S}_{n,r}}$ and $\pr_{\rpart}$ to denote probabilities in $\mathfrak{S}_{n,r}$ and $C_{\rpart}$ (with respect to the uniform distribution on these sets). We similarly use $\E_{\mathfrak{S}_{n,r}}$ and $\E_{\rpart}$ for the expected values on the corresponding probability spaces.

\subsection{Hyperoctahedral groups}\label{sec:Bn_preliminaries}

We now give an overview of the hyperoctahedral group $B_n\cong \mathfrak{S}_{n,2}$, including statistics that will be used as illustrative examples in the next section. The conventions introduced in this section and used throughout this paper are intended to align with the literature on $B_n$.

Let $[\pm n]$ denote the set $\{\pm 1,\pm 2,\ldots,\pm n\}$. The hyperoctahedral group $B_n$ is the group of permutations $\omega$ on $[\pm n]$ satisfying $\omega(i)=-\omega(-i)$. A comparison with the definitions in the preceding sections shows that this is isomorphic to the $2$-colored permutation group $\mathfrak{S}_{n,2}$, where the set $\{1,2,\ldots,n\}$ corresponds to elements colored $0$ and the set $\{-1,-2,\ldots,-n\}$ corresponds to elements colored $1$.

Similar to general colored permutation groups, it suffices to specify the images of the set $[n]$ to determine an element in $B_n$. Using these values, elements in $B_n$ can again be represented using the two-line, one-line, and cycle notations. In the cycle notation, a cycle is \emph{even} (resp. \emph{odd}) if there is an even (resp. odd) number of negatives in the cycle. These respectively correspond to cycles with color $0$ and $1$ in Definition \ref{def:cycles}. We give examples below.

\begin{ex}\label{ex:basic1}
Consider the following element of $B_8$, written in two-line and one-line notation. 
\begin{equation*}
    \begin{split}
        \omega & =
\left(\begin{matrix}
1 & 2 & 3 & 4 & 5 & 6 & 7 & 8 \\
+2 & +7 & -1 & -5 & +8 & +3 & +6 &-4 
\end{matrix}\right) 
 =[2,7,-1,-5, 8,3,6,-4].
    \end{split}
\end{equation*}
Its cycle notation is given by $(-5,8,-4)(2, 7, 6, 3,-1),$ where we include commas for clarity. The two cycles are respectively an even cycle of length 3 and an odd cycle of length 5. Thus, the cycle type of $\omega$ is $(\lambda,\mu)=((3),(5))$. We have $m_3(\lambda)=m_5(\mu)=1$, and all other $m_i(\lambda)$ and $m_i(\mu)$ are $0$.
\end{ex}

\begin{ex}\label{ex:basic2} 
Now consider 
\[\omega=
\left(\begin{matrix}
1 & 2 & 3 & 4 & 5 & 6 & 7 & 8 \\
+2 & +7 & 3 & -4 & +8 & -1 & -6 &-5
\end{matrix}\right)
=[2,7,3,-4, 8,-1,-6,-5].\]
Its even cycles are 
$(2,7,-6,-1)$ and $(3) $, respectively of lengths 4 and 1, and its odd cycles are $(-4)$ and $(8, -5)$, respectively of lengths 1 and 2.
Hence the cycle type is 
$(\lambda,\mu)=((4,1), (2,1))$. We have ${m_1(\lambda)}=m_4(\lambda)=1$ and ${m_1(\mu)}={m_2(\mu)}=1.$ 
\end{ex}

From Proposition \ref{cycletype}, the conjugacy classes in $B_n$ are uniquely determined by a 
bi-partition $(\lambda,\mu)$ of $n$, where $\lambda$  corresponds to the even cycles and $\mu$ to the odd cycles. Throughout this paper, $C_{\lambda,\mu}$ denotes the conjugacy class of $B_n$ indexed by the ordered pair $(\lambda, \mu)$. Additionally, $\pr_{B_n}$ and $\pr_{\lambda,\mu}$ denote probabilities in $B_n$ and $C_{\lambda,\mu}$ (with respect to the uniform distribution on these sets), and $\E_{B_n}$ and $\E_{\lambda,\mu}$ denote expectations on these sets.

The size of any conjugacy class in $B_n$ is well 
 known. For an integer partition $\lambda\vdash k$, let $\ell(\lambda)$ denote the number of parts of $\lambda$, so that 
$\ell(\lambda)=\sum_i  m_i(\lambda)$ and $k=\sum_i  i \cdot m_i(\lambda)$.  Let $z_\lambda$ denote the order of the centralizer of an element of cycle type $\lambda$ in the symmetric group $\mathfrak{S}_k$, i.e., $z_\lambda=\prod_i i^{m_i(\lambda)} m_i(\lambda)!$. The size of the conjugacy class $C_{\lambda}$ in $\mathfrak{S}_n$ is given by $|\mathfrak{S}_n|/z_{\lambda}$, and an analogous result holds for $B_n$. 

\begin{definition}\label{def:doublezee} For any bi-partition $(\lambda,\mu)$ of $n$, define $z_{\lambda,\mu}$ to be the product ${2^{\ell(\lambda)} z_\lambda 2^{\ell(\mu)} z_\mu}$. 
\end{definition}

\begin{lemma}\label{lem:classsize}\cite[4.1.33, Lemma 4.2.10]{JamesKerber1981}
Let $C_{\lambda,\mu}$ be the conjugacy class of $B_n$ indexed by the bi-partition $(\lambda,\mu)$ of n.
The order of the centralizer of an element of cycle type $(\lambda, \mu)$ is $z_{\lambda,\mu}$, and the size of the conjugacy class is

\[|C_{\lambda,\mu}|=\frac{n!\cdot 2^n}{2^{\ell(\lambda)} z_\lambda 2^{\ell(\mu)} z_\mu}=\frac{|B_n|}{z_{\lambda,\mu}}.\]

\end{lemma} 

We will be particularly interested in two statistics on $B_n$ related to its presentation as a Coxeter group; see \cite{BB} for a detailed account. The signed permutation group $B_n$ can be generated using the elements $s_0=[-1,2,3,\ldots,n]=(-1)$ and $s_i=[1,\ldots,i-1,i+1,i,i+2,\ldots,n]=(i, i+1)$ for $1\leq i\leq n-1$.  With respect to these generators,  the inversion (or length) statistic and the descent statistic are given by
\[\inv_B(\omega)=\min\{\ell \mid \omega=s_{i_1}\dots s_{i_\ell} \text{ for some $i_1,\ldots,i_{\ell}\in \{0\}\cup [n-1]$}\},\]
\[\des_B(\omega)=\left|\{i\in \{0\}\cup [n-1] \mid \inv_B( \omega s_i )<\inv_B(\omega)\}\right|.\]

One can compute these statistics in alternative ways. For $\des_B$, we will primarily use the formulation given in \cite[Proposition 8.1.2]{BB}:
\[\des_B(\omega)=|\{i\in \{0\}\cup [n-1]\mid \omega(i)>\omega(i+1)\}|\]
with the convention that $\omega(0)=0$. Note that the condition $\omega(i)>\omega(i+1)$ is with respect to the usual ordering on $[\pm n]\cup \{0\}$. For $\inv_B$, we first define two additional statistics
\[\inv(\omega)=|\{(i,j)\in [n]\times [n] \mid i<j \text{ and }\omega(i)>\omega(j)\}|,\]
\[\negative(\omega)= \sum_{i\in [n],\omega(i)<0} \omega(i).\]
\noindent  Now  \cite[Proposition 8.1.1]{BB} shows that $\inv_B$ can be calculated using
\begin{equation}\label{eqn:InvA-InvB-Neg}\inv_B(\omega)=\inv(\omega)-\negative(\omega).
\end{equation}

\section{Partial colored permutations and moments of statistics}\label{section:colored}

In this section, we will introduce partial colored permutations and define what it means for a colored permutation statistic to have degree at most $m$. These directly generalize the definitions in \cite{GRWCPermutationStatistics}. Using partial colored permutations, we prove Theorem~\ref{thm:MainIndependence} and Theorem~\ref{thm:Polynomiality}.

\subsection{Partial colored permutations}

We begin with a definition of partial colored permutations. Similar to how colored permutations consist of two components, a partial colored permutation will also consist of two components.

\begin{definition}\label{def:coloredconstraint}
    A \emph{partial colored permutation} on $\mathfrak{S}_{n,r}$ is a pair $(K,\kappa)$, where $K=\{(i_h,j_h)\}_{h=1}^m$ consists of $m$ distinct ordered pairs of elements in $[n]$, $\{i_h\}_{h=1}^m$ are $m$ distinct elements, $\{j_h\}_{h=1}^m$ are $m$ distinct elements, and $\kappa:\{j_1,\ldots,j_m\}\to \mathbb{Z}_r$ is any function. We will sometimes represent $\kappa$ also using ordered pairs $\{(j_h,\kappa(j_h))\}_{h=1}^m$. We call $m$ the \emph{size} of $(K,\kappa)$. For brevity, we will alternatively denote a partial colored permutation using a single set of ordered pairs \[(K,\kappa)=\left\{\left(i_h^0,j_h^{\kappa(j_h)}\right)\right\}_{h=1}^m\]
    of elements in $[n]^r$, and we sometimes omit set brackets when  $m=1$.
\end{definition}

\begin{remark}
In previous work \cite{GRWCPermutationStatistics, GRWCFPSAC}, we referred to partial (colored) permutations as \emph{(colored) permutation constraints}. The current conventions are intended to better align with the literature.
\end{remark}

We now define what it means for a colored permutation to satisfy a partial colored permutation. This allows us to view partial colored permutations as specifying some of the function values in a colored permutation.

\begin{definition}
    Let $(K,\kappa)$ be a partial colored permutation on $\mathfrak{S}_{n,r}$ of size $m$, and let $(\omega,\tau)\in \mathfrak{S}_{n,r}$ be a colored permutation. The permutation $\omega\in \mathfrak{S}_n$ \emph{satisfies} $K$ if $\omega(i_h)=j_h$ for all $h\in [m]$. The coloring $\tau:[n]\to \mathbb{Z}_r$ \emph{satisfies} $\kappa$ if $\tau(j_h)=\kappa(j_h)$ for all $h\in [m]$. The colored permutation $(\omega,\tau)$ \emph{satisfies} $(K,\kappa)$ if $\omega$ satisfies $K$ and $\tau$ satisfies $\kappa$. 
\end{definition}

For example, the colored permutation $(\omega,\tau)$ of Example~\ref{ex:coloredpermutation} satisfies the size 3 partial colored permutation $(K, \kappa)=\{(1^0, 4^1), (4^0, 3^1), (5^0, 2^1)\}$, but does not satisfy the size 2 partial colored permutation $\{(1^0, 4^2), (4^0, 3^1)\}$. More generally, a given colored permutation $(\omega,\tau)$ satisfies exactly $\binom{n}{m}$ partial colored permutations of size $m$.

Recall that a (real) colored permutation statistic is a function $X: \mathfrak{S}_{n,r} \rightarrow \mathbb{R}$. For any partial colored permutation $(K,\kappa)$, one useful statistic will be the indicator function $I_{(K,\kappa)}:\mathfrak{S}_{n,r}\to \mathbb{R}$ that takes value $1$ on colored permutations satisfying $(K,\kappa)$ and $0$ otherwise. We will be interested in decompositions of colored permutation statistics in terms of these indicator functions.

\begin{definition}\label{def:degree}
    Define $\mathcal{C}_m$ to be the set of partial colored permutations with size at most $m$. A colored permutation statistic $X:\mathfrak{S}_{n,r}\to \mathbb{R}$ has \emph{degree} $m$ if $X$ is in the (real) vector space spanned by $\{I_{(K,\kappa)}:(K,\kappa)\in \mathcal{C}_m\}$ and $X$ is not in the (real) vector space spanned by $\{I_{(K,\kappa)}:(K,\kappa)\in \mathcal{C}_{m-1}\}$. 
\end{definition}

Many colored permutation statistics have natural decompositions that allow one to bound their degree. We give some examples in $B_n$ using the ordered pair notation from Definition~\ref{def:coloredconstraint} and the usual conventions on $B_n$ of using the elements $[\pm n]$ in place of $[n]^2$. 

\begin{ex}\label{ex:des-inv-negsum-constraints}
For the statistics defined on $B_n$ given in Section~\ref{sec:Bn_preliminaries}, we have
        \[\des_B =\sum_{j\in [n]} I_{(1,-j)}+\sum_{i\in [n-1]} \,\sum_{\substack{k,\ell\in [\pm n]\\ k<\ell}} I_{\{(i,\ell),(i+1,k)\}},\]
        \[\inv=\sum_{ \substack{i,j \in [n]
     \\ i<j}} \, \sum_{\substack{k,\ell\in [\pm n] \\ k<\ell}} \, I_{\{(i,\ell),(j,k)\}},\]
        \[\negative =\sum_{i\in [n]} \, \sum_{j\in [n]} (-j)I_{(i, -j)}.\]
This shows that $\des_B$ and $\inv$ have degree at most $2$, and $\negative$ has degree $1$. Since $\inv_B$ is the difference of $\inv$ and $\negative$, we also see that $\inv_B$  has degree at most $2$. 
\end{ex}

Decompositions like the ones in Example~\ref{ex:des-inv-negsum-constraints} result in upper bounds on the degree of a statistic. However, determining the degree of a statistic can be difficult. Due to this, our results will be stated for statistics that have degree \emph{at most} $m$.

\begin{remark}
    Note that a colored permutation $(\omega,\tau)\in \mathfrak{S}_{n,r}$ is itself a partial colored permutation of size $n$. Consequently, we can express any statistic $X$ as
    \[X=\sum_{(\omega,\tau)\in \mathfrak{S}_{n,r}} X(\omega,\tau)\cdot I_{(\omega,\tau)}.\]
    This shows that every colored permutation statistic on $\mathfrak{S}_{n,r}$ has degree at most $n$.
\end{remark}

\begin{remark}
If $I_{(K,\kappa)}$ is the indicator function for a partial colored permutation $(K,\kappa)=\{(i_h^0,j_h^{\kappa(j_h)})\}_{h=1}^m$ on $\mathfrak{S}_{n,r}$ of size $m < n$, then we can select some $i\in [n]\setminus \{i_h\}_{h=1}^m$ and express
\begin{equation*}
I_{(K,\kappa)}=\sum_{j\in [n]\setminus \{j_h\}_{h=1}^m }\, \sum_{c\in \mathbb{Z}_r} I_{(K\cup \{(i,j)\},\kappa\cup \{(j,c)\})}.
\end{equation*}
    Because of this, one can modify $\mathcal{C}_m$ in Definition \ref{def:degree} to contain only the partial colored permutations of size exactly $m$ without affecting the definition of degree for a colored permutation statistic. We have chosen the current convention to simplify proofs. 
\end{remark}

\subsection{Moments on conjugacy classes without short cycles}
 In this section, we will prove Theorem~\ref{thm:MainIndependence}, adapting the general strategy of \cite[Section 7]{GRWCPermutationStatistics}. It will be convenient to consider the set of pairs $K$ in a partial colored permutation $(K,\kappa)$ as a graph, so we formalize this now. 

\begin{definition}
Let $(K, \kappa)$ be a partial colored permutation in $\mathfrak{S}_{n,r}$, with $K = \{ (i_{1}, j_{1}) ,\ldots, (i_{m}, j_{m})\}$. The \emph{graph} of $(K, \kappa)$, denoted $G(K, \kappa)$, is the directed graph with vertex set $V = [n]$ and directed edge set $K$. Note that $G(K, \kappa)$ depends only on $K$ and not on $\kappa$.
\end{definition}

\begin{definition}\label{def:acyclic}
A partial colored permutation $(K,\kappa)$ on $\mathfrak{S}_{n,r}$ is \emph{acyclic} if the graph of $G(K, \kappa)$ does not contain any cycles. Observe that in this case, $G(K, \kappa)$ consists of a set of directed paths, each of which has a source vertex, i.e., a vertex of in-degree 0. Additionally, acyclicity implies that the total number of edges in $G(K,\kappa)$ is $|K|<n$.
\end{definition}

We next show that on any conjugacy class without short cycles, the mean of $I_{(K,\kappa)}$ takes one of two values determined entirely by acyclicity of $(K,\kappa)$. Note that since $I_{(K,\kappa)}$ only takes values in $\{0,1\}$, its mean on any $C_{\rpart}$ can be expressed as 
\[\E_{\rpart}[I_{(K,\kappa)}]=\pr_{\lambda}[I_{(K,\kappa)}=1] = \pr_{\lambda}[(\omega,\tau) \text{ satisfies } (K,\kappa)].\]
The next two lemmas explicitly calculate the probability on the right side. Note that the first lemma generalizes \cite[Lemma 7.15]{GRWCPermutationStatistics} to the colored permutation setting, and our proof follows the general approach given in \cite{GRWCPermutationStatistics} but with several nontrivial technical modifications to properly account for color. 

\begin{lemma}\label{lem:omegasatisfiesK}(Compare to $\mathfrak{S}_{n}$, \cite[Lemma 7.15]{GRWCPermutationStatistics})
    Let $(K,\kappa)$ be a partial colored permutation on $\mathfrak{S}_{n,r}$ of size $m$, and let $C_{\rpart}$ be a conjugacy class of $\mathfrak{S}_{n,r}$ with no cycles of length $1,2,\ldots,m$.
    If $K$ is not acyclic, then \[\pr_{\rpart}[\omega \mathrm{\ satisfies\ } K]=0.\]
    If $K$ is acyclic, then 
    \[
\pr_{\rpart}[\omega \mathrm{\ satisfies\ } K] = \frac{1}{(n-1)(n-2) \cdots (n-m)}.
\]
\end{lemma}

We emphasize that in this lemma, $\kappa$ is not relevant, so we do not use it. Instead, $\kappa$ will be considered in Lemma~\ref{lem:HigherMomentsIndependence}.

\begin{proof}[Proof of Lemma~\ref{lem:omegasatisfiesK}]
 We first consider the case when $K$ is not acyclic. In this case, in order for an element $(\omega,\tau)\in C_{\rpart}$ to have the property that $\omega$ satisfies $K$, it must be that $\omega$ contains a cycle induced by the conditions in $K$. Since $K$ has size $m$, this cycle has length at most $m$.  However, we assumed $C_{\rpart}$ has no cycles of length $1,2,\ldots,m$, so this is not possible. Hence, no such elements exist, and we conclude that  $\pr_{\rpart}[\omega \text{\ satisfies\ } K] = 0$. 

    For the case when $K$ is acyclic, we use induction on $n$ and $m$ to compute the probability that a colored permutation of cycle type $\rpart$ chosen uniformly at random satisfies $K$. As base cases, consider when $n$ is arbitrary and $m = 1$.  We express $K=\{(i,j)\}$ and analyze $\pr_{\rpart}[\omega(i) = j]$. Consider any $k \in [n] \setminus  \{i,j\}$. 
    Letting $\mathbf{0}:[n]\to \mathbb{Z}_r$ denote the zero coloring, conjugating by the colored permutation $((j,k),\mathbf{0})$ induces bijections between \[\{(\omega,\tau)\in C_{\rpart}\mid \omega(i)=j\} \quad \text{ and } \quad  \{(\omega,\tau)\in C_{\rpart}\mid \omega(i)=k\}.\]
    Therefore, $\pr_{\rpart}[\omega(i) = k]$ is invariant under our choice of $k\neq i$. As $C_{\rpart}$ contains no cycles of length $1$, we have that $\pr_{\rpart}[\omega(i) = k]=0 $ when $k=i$, and combined, we see that
\[
\pr_{\rpart}[\omega(i) = j] = \frac{1}{n-1}.
\]
\noindent This establishes our base case of $m = 1$ when $n$ is arbitrary. 

Now fix $n> m>1$, and suppose that the result holds for $n-1$ and $m-1$. Consider an acyclic partial colored permutation $(K,\kappa)$ on $\mathfrak{S}_{n,r}$ of size $m$ and a conjugacy class $C_{\rpart}$ of $\mathfrak{S}_{n,r}$ without cycles of length $1,2,\ldots m$. Express $K = \{ (i_{1}, j_{1}), \ldots, (i_{m}, j_{m})\}$, and partition $C_{\rpart}$ into $\bigsqcup_{k=1}^n \bigsqcup_{c\in \mathbb{Z}_r} \Omega_{k,c}$, 
 where \[\Omega_{k,c}=\{(\omega,\tau)\in C_{\rpart}\mid i_m \text{ appears in a cycle of length $k$ and color $c$}\}.\]
Using conditional expectations, the law of total probability, and our result in the preceding paragraph, we express $\pr_{\rpart}[\omega \text{ satisfies } K]$ as
\begin{equation}\label{eq:omegasatisfiesK}
\begin{split}
     &\,\pr_{\rpart} \left[\bigcap_{h=1}^{m} \{\omega(i_{h}) = j_{h}\} \right] \\
= & \,\pr_{\rpart} \left[\bigcap_{h=1}^{m-1}  \{\omega(i_{\ell}) = j_{h}\} \biggr| \omega(i_{m}) = j_{m} \right] \cdot \pr_{\rpart}[\omega(i_{m}) = j_{m}] \\
 =& \,\frac{1}{n-1} \cdot \sum_{k=1}^{n} \,\sum_{c\in\mathbb{Z}_r}\pr_{\rpart} \left[ \bigcap_{h=1}^{m-1}  \{\omega(i_{h}) = j_{h}\} \cap \Omega_{k,c} \,\biggr|\, \omega(i_{m}) = j_{m} \right].
 \end{split}
\end{equation}
Each summand above can then be expressed as 
\begin{equation}\label{eq:omegasatisfiesK2}
    \pr_{\rpart}\left[ \bigcap_{h=1}^{m-1} \{\omega(i_{h}) = j_{h}\} \,\biggr|\, \Omega_{k,c} \cap \{\omega(i_{m}) = j_{m}\} \right] \cdot \pr_{\rpart}[\Omega_{k,c} \mid \omega(i_{m}) = j_{m}].
\end{equation}

We now consider the acyclic partial colored permutation $K' = \{ (i_{1}, j_{1}), \ldots, (i_{m-1}, j_{m-1})\}$ of size $m-1$. We can assume without loss of generality that $i_m$ is a source vertex in $G(K,\kappa)$, so $i_m$ is not an element in any ordered pair in $K'$. For any fixed $k\in [n]$ and $c\in \mathbb{Z}_r$ where $\pr_{\rpart}[\Omega_{k,c}]\neq 0$, define $\rpart'$ to be the $r$-partition of $[n-1]$ obtained by starting with $\rpart$ and replacing a part of size $k$ and color $c$ with a part of size $k-1$ and color $c$. Since $\rpart$ contains only partitions whose parts all have lengths larger than $m$, $\rpart'$ contains only partitions whose parts all have lengths larger than $m-1$. Considering $C_{\rpart'}$ as a conjugacy class of the colored permutation group on $r$ copies of the elements in $[n]\setminus \{i_m\}$, define a function
\[\pi:\{(\omega,\tau)\in \Omega_{k,c}\mid \omega(i_m)=j_m\}\to C_{\rpart'}\]
that in the cycle notation of $(\omega,\tau)$ replaces $i_m^{\tau(i_m)} j_m^{\tau(j_m)}$ 
with $j_m^{\tau(i_m)+\tau(j_m)}$. Observe that $\pi$ is an $r$-to-$1$ map from $C_{\rpart}$ to $C_{\rpart'}$. Additionally, for a colored permutation $(\omega,\tau)$ in the domain with image $(\omega',\tau')\in C_{\lambda'}$, we have that $\omega$ satisfies $K'$ if and only if $\omega'$ satisfies $K'$. In particular, if we let $(\omega',\tau')\in C_{\rpart'}$ be generated uniformly at random, then this observation combined with our induction hypothesis implies that for each fixed $k$ and $c$,
\begin{equation*}
\begin{split}
& \pr_{\rpart} \left[ \bigcap_{h=1}^{m-1} \{\omega(i_{h}) = j_{h}\} \, \biggr| \, \Omega_{k,c} \cap \{\omega(i_{m}) = j_{m})\} \right] \\
=& \, \pr_{\rpart'} \left[\bigcap_{h=1}^{m-1} \{\omega'(i_{h}) 
= j_{h} \}\right] \\
=& \, \frac{1}{(n-2)(n-3) \cdots (n-m)}.
\end{split}
\end{equation*}
Note that the first term is $(n-1)-1$ and the last term is $(n-1)-(m-1)$ since this probability involves $\mathfrak{S}_{n-1,r}$ and a partial colored permutation of size $m-1$. Returning to \eqref{eq:omegasatisfiesK} and \eqref{eq:omegasatisfiesK2}, we conclude
\begin{align*}
& \pr_{\rpart}\left[ \bigcap_{k=1}^{n} \{\omega(i_{k}) = j_{k}\} \right]  \\
= & \,\frac{1}{n-1}\sum_{k=1}^{n} \,\sum_{c\in \mathbb{Z}_r} \frac{1}{(n-2)(n-3)\cdots (n-m)} \cdot \pr_{\rpart}[\Omega_{k,c} \mid \omega(i_{m}) = j_{m}]\\
= &\, \frac{1}{(n-1)(n-2)\cdots (n-m)} \cdot \sum_{k=1}^{n} \,\sum_{c\in \mathbb{Z}_r}\pr_{\rpart}[\Omega_{k,c} \mid \omega(i_{m}) = j_{m}]\\
=& \, \frac{1}{(n-1)(n-2) \cdots (n-m)}. \qedhere 
\end{align*} 
\end{proof}

\begin{lemma} \label{lem:HigherMomentsIndependence}
Let $(K, \kappa)$ be a partial colored permutation on $\mathfrak{S}_{n,r}$ of size $m$, and let $C_{\rpart}$ be a conjugacy class of $\mathfrak{S}_{n,r}$ with no cycles of length $1,2,\ldots,m$. \noindent If $K$ is not acyclic, then
\[
\pr_{\rpart}[(\omega, \tau) \text{\ satisfies\ } (K, \kappa)] = 0.
\]
If $K$ is acyclic, then
\[
\pr_{\rpart}[(\omega, \tau) \text{\ satisfies\ } (K, \kappa)] = \frac{1}{(n-1)(n-2) \cdots (n-m)} \cdot \frac{1}{r^{m}}.
\]
\end{lemma}

\begin{proof}
    First, we express
    \begin{equation*}
        \pr_{\rpart}[(\omega, \tau) \text{ satisfies } (K, \kappa)] = \pr_{\rpart}[\omega \text{ satisfies } K] \cdot \pr_{\rpart}[\tau \text{ satisfies } \kappa \mid \omega \text{ satisfies } K].
    \end{equation*}
    After applying Lemma~\ref{lem:omegasatisfiesK} to the first term on the right side, it now suffices to show that when $(K,\kappa)$ is acyclic, the second term is $1/r^m$.

    Express $K=\{(i_h,j_h)\}_{h=1}^m$. Define an action of $\mathbb{Z}_r^m$ on $C_{\rpart}$ as follows: the $m$-tuple $e_h$ with 0 everywhere and $1$ in position $h$ acts on $(\omega,\tau)$ by
    \begin{itemize}
        \item adding $1$ to the color $\tau(j_h)$, and
        \item subtracting $1$ from the color $\tau(i)$, where $i\in [n]\setminus \{j_1,\ldots,j_m\}$ is the smallest element that appears in the cycle containing $j_h$.
    \end{itemize}
    Since $C_{\rpart}$ contains no cycles of length $1,2,\ldots,m$, the element $i$ always exists. It is straightforward to see that extending this to all elements in $\mathbb{Z}_r^m$ results in a well-defined group action on $C_{\rpart}$, and each orbit has size $r^m$. It is clear that each orbit contains an element $(\omega,\tau)$ such that $\tau$ satisfies $\kappa$, and since the action of any nonzero element in $\mathbb{Z}_r^m$ on $(\omega,\tau)$ results in a colored permutation $(\omega,\tau')$ where $\tau'$ does not satisfy $\kappa$, we conclude that exactly one element in each orbit has this property. Furthermore, the subset $\{(\omega,\tau)\in C_{\rpart} \mid \omega \text{ satisfies } K\}$ is invariant under this action, and this allows us to conclude the desired result that 
    \[\pr_{\rpart}[\tau \text{ satisfies } \kappa \mid \omega \text{ satisfies } K]=\pr_{\rpart}[\tau \text{ satisfies } \kappa]=\frac{1}{r^m}. \qedhere\]
 \end{proof}

 Our proof in the preceding lemma also implies the following result, namely that on conjugacy classes without short cycles, satisfying $K$ and satisfying $\kappa$ are independent. 

\begin{corollary}
    Let $(K,\kappa)$ be a partial colored permutation  on $\mathfrak{S}_{n,r}$ of size $m$, and let $C_{\rpart}$ be a conjugacy class of $\mathfrak{S}_{n,r}$ with no cycles of length $1,2,\ldots,m$. Then
    \[\pr_{\rpart}[(\omega, \tau) \text{ satisfies } (K, \kappa)] = \pr_{\rpart}[\omega \text{ satisfies } K] \cdot \pr_{\rpart}[\tau \text{ satisfies } \kappa].\]
\end{corollary}

\begin{remark}
    Our preceding results also imply that when $K$ has size $m$ and $C_{\rpart}$ has no cycles of length $1,2,\ldots,m$,
    \[\pr_{\rpart}[\omega \text{ satisfies } K]=\pr_{\rpart}[\omega \text{ satisfies } K\mid \tau \text{ satisfies }\kappa]= \frac{1}{(n-1)(n-2)\ldots(n-m)}.\]
    While one can attempt to show the second equality directly, we found this to be much more technical than the proof of Lemma~\ref{lem:omegasatisfiesK}.
\end{remark}

Lemma~\ref{lem:HigherMomentsIndependence} allows us to make statements about the mean of statistics on conjugacy classes without short cycles, but our results involve arbitrary moments. To connect these, we will need to analyze products of these indicator functions $I_{(K,\kappa)}$. The terminology below is adapted from \cite{hamaker2022characters}.

\begin{definition}
Two partial colored permutations $(K_{1}, \kappa_{1})$ and $(K_{2}, \kappa_{2})$ on $\mathfrak{S}_{n,r}$ are \textit{compatible} if there exists a colored permutation $(\omega, \tau) \in \mathfrak{S}_{n,r}$ satisfying both partial colored permutations.
\end{definition}

Note that when two partial colored permutations are compatible, we can use the unions $K_1\cup K_2$ and $\kappa_1\cup \kappa_2$ to define a new partial colored permutation, and the size of $(K_1\cup K_2,\tau_1\cup \tau_2)$ is bounded by the sum of the sizes of $(K_1,\tau_1)$ and $(K_1,\tau_2)$. Using this, the following results are straightforward exercises.

\begin{lemma} \label{lem:HRLem4.2}
Let $(K_{1}, \kappa_{1})$ and $(K_{2}, \kappa_{2})$ be two partial colored permutations  on $\mathfrak{S}_{n,r}$. If $(K_{1}, \kappa_{1})$ and $(K_{2}, \kappa_{2})$ are not compatible, then $I_{(K_{1}, \kappa_{1})} \cdot I_{(K_{2}, \kappa_{2})}$ is identically  zero. If $(K_{1}, \kappa_{1})$ and $(K_{2}, \kappa_{2})$ are compatible, then $I_{(K_{1}, \kappa_{1})} \cdot I_{(K_{2}, \kappa_{2})} = I_{(K_1 \cup K_2, \kappa_1 \cup \kappa_2)}$.
\end{lemma}

\begin{corollary} \label{cor:Thm7.16}
Suppose $X_{1}$ and $X_{2}$ are statistics on $\mathfrak{S}_{n,r}$ with degrees at most $m_1$ and $m_2$, respectively. Then $X_1\cdot X_2$ has degree at most $m_1+m_2$. In particular, for any integer $k\geq 1$ such that $m_1k\leq n$, we have that $X_1^k$ has degree at most $m_1k$.
\end{corollary}

By combining these two facts, we can now establish one of our main results.

\begin{theorem}[Theorem~\ref{thm:MainIndependence}] \label{thm:IndependenceHigherMoments}
    Suppose $X:\mathfrak{S}_{n,r}\to \mathbb{R}$ has degree at most $m$. For any $k\geq 1$, the $k$th moment $\E_{\rpart}[X^k]$ coincides on all conjugacy classes $C_{\rpart}$ with no cycles of length $1,2,\ldots,mk$.
\end{theorem}

\begin{proof}
    We first consider the case when $k=1$ and $C_{\lambda}$ has no cycles of length $1,2,\ldots,m$. Express $X=\sum_{(K,k)\in \mathcal{C}} c_{(K,\kappa)} I_{(K,\kappa)}$ where $c_{(K,\kappa)}\in \mathbb{R}$ and each $(K,k)\in \mathcal{C}$ has size at most $m$. Linearity of expectation implies
    \begin{equation}\label{eq:firstmoment}
        \E_{\rpart}[X]=\sum_{(K,\kappa)\in \mathcal{C}} c_{(K,\kappa)}\cdot \E_{\rpart}[I_{(K,\kappa)}].
    \end{equation}
    Applying Lemma \ref{lem:HigherMomentsIndependence} to each summand, either $\E_{\rpart}[I_{(K,\kappa)}]=0$ or 
    \begin{equation}\label{eq:indicator_mean}
        \E_{\rpart}[I_{(K,\kappa)}]=\frac{1}{(n-1)(n-2)\dots (n-|K|)}\cdot \frac{1}{r^{|K|}},
    \end{equation}
    the latter case corresponding to acyclicity of $(K,\kappa)$. Acyclicity of $(K,\kappa)$ is independent of $C_{\rpart}$, so (\ref{eq:firstmoment}) is equivalent to the following expression that is independent of $C_{\rpart}$:
    \[\sum_{\substack{(K,\kappa)\in \mathcal{C} \\ \text{acyclic}}}\frac{c_{(K,\kappa)}}{(n-1)(n-2)\dots (n-|K|)}\cdot \frac{1}{r^{|K|}}.\]

    For higher moments, it suffices to consider when $mk< n$. Corollary \ref{cor:Thm7.16} implies that if $X$ is has degree at most $m$, then $X^k$ has degree at most $mk$. The general result for $\E_{\lambda}[X^k]$ now follows by combining this with the above result. 
\end{proof}

\begin{corollary}[Compare to $\mathfrak{S}_{n}$, {\cite[Corollary~7.17]{GRWCPermutationStatistics}}]
Let $X : \mathfrak{S}_{n,r} \to \mathbb{R}$ be a colored permutation statistic. Suppose that for some $j\in \mathbb{Z}_r$, $k\geq 1$, and $i\geq mk+1$, the moment $\mathbb{E}_{\rpart}[X^k]$ depends on $m_i(\lambda^j)$. Then $X$ cannot have degree $k$. 
\end{corollary}

We note that the argument in Theorem \ref{thm:IndependenceHigherMoments} has practical applications for computing moments of statistics on those conjugacy classes. We give an example below.

\begin{ex}\label{rem:negsum}
    Express $\negative$ on $B_n$ as 
    \[\negative = \sum_{i\in [n],k\in [n]} (-k)\cdot I_{(i,-k)}.\]
    Note that $(i,-i)$ is not acyclic, while $(i,-k)$ with $k\neq i$ is acyclic. 
    Applying Equations (\ref{eq:firstmoment}) and (\ref{eq:indicator_mean}), we see that on any conjugacy classes with no cycles of length $1$,
    \begin{equation*}
        \begin{split}
            \E_{\lambda,\mu}[\negative] & =-\frac{1}{(n-1)\cdot 2}\cdot \sum_{i\in [n],k\in [n]\setminus i} k \\
            & = -\frac{1}{(n-1)\cdot 2}\cdot \sum_{i\in [n]} \left(\binom{n+1}{2}-i \right)\\
            & =-\frac{1}{2}\binom{n+1}{2}.
        \end{split}
    \end{equation*}
    In general, Theorem~\ref{thm:IndependenceHigherMoments} shows that calculating the $k$th moment of a statistic on conjugacy classes without short cycles corresponds to calculating weighted sums of acyclic partial colored permutations.  
\end{ex}

\subsection{Order-invariant colored permutation statistics}

For some statistics $(X_n)_{n\geq 1}$ defined on $(\mathfrak{S}_{n,r})_{n\geq 1}$, the enumeration in Example~\ref{rem:negsum} can be used to study general properties of the $k$th moment of $(X_n)_{n\geq 1}$ on conjugacy classes without short cycles. For example, in some cases, this moment is a polynomial in $n$. We describe one family of statistics where this holds.

\begin{definition}
    Let $(K,\kappa)$ be a partial colored permutation on $\mathfrak{S}_{n,r}$ with $K=\{(i_h,j_h)\}_{h=1}^m$. The \textit{support} of $(K, \kappa)$ is the set of elements \[\text{supp}(K, \kappa) = \{ i_{1}, \ldots, i_{m}, j_{1}, \ldots, j_{m}\}.\] We emphasize that $\supp(K, \kappa)$ is a set and not a multiset. 
\end{definition}

\begin{definition}
Let $(K,\kappa)$ be a partial colored permutation on $\mathfrak{S}_{n,r}$ with $K=\{(i_h,j_h)\}_{h=1}^m$. For any order-preserving injection $f : \text{supp}(K, \kappa) \to [n]$, define $f(K,\kappa)\in \mathfrak{S}_{n,r}$ to be the partial colored permutation \[f(K,\kappa)=(\,\{(f(i_1),f(j_1)),\ldots,(f(i_m),f(j_m))\},\,\{(f(j_1),\kappa(j_1)),\ldots, (f(j_m),\kappa(k_m))\}\,).\]
\end{definition}

\begin{definition}
    A set of partial colored permutations $\mathcal{C}$ is \emph{order-invariant} if all partial colored permutations in $\mathcal{C}$ have the same size, and for any $(K,\kappa)\in \mathcal{C}$ and any order-preserving $f:\operatorname{supp}(K,\kappa)\to [n]$, we have $f(K,\kappa)\in \mathcal{C}$. A statistic $X$ is \emph{order-invariant} if it has the form $X=\sum_{(K,\kappa)\in \mathcal{C}} I_{(K,\kappa)}$ for some order-invariant $\mathcal{C}$. In this case, we say $X$ is \emph{induced} by $\mathcal{C}$.
\end{definition}

\begin{remark}
In previous work \cite{GRWCPermutationStatistics, GRWCFPSAC}, we referred to order-invariant extensions as \emph{symmetric extensions}.
\end{remark}


Many statistics on $\mathfrak{S}_{n,r}$ are naturally order-invariant. We give an example below in $B_n$.

\begin{ex}
    Consider the statistic $\inv$ on $B_n$, which can be expressed as
    \[\inv=\sum_{ \substack{i,j \in [n]
     \\ i<j}} \, \sum_{\substack{k,\ell\in [\pm n] \\ k<\ell}} \, I_{\{(i,\ell),(j,k)\}}.\]
    Each partial colored permutation  $\{(i,\ell),(j,k)\}$ above has size $2$, and we denote this set of partial colored permutations as $\mathcal{C}$. In the case that $k,\ell>0$, then for any order-preserving $f:\{i,j,k,\ell\}\to [n]$, we see that $\{(f(i),f(\ell)),(f(j),f(k))\}\in \mathcal{C}$. Note that the set $\{i,j,k,\ell\}$ need not consist of four distinct elements. In the case where $k<0$ and $\ell>0$, we see that for any order-preserving $f:\{i,j,|k|,\ell\}\to [n]$, we have $\{(f(i),f(\ell)),(f(j),-f(|k|))\}\in \mathcal{C}$. The same argument holds for the case $k,\ell<0$.
\end{ex}

Given an order-invariant statistic on a colored permutation group $\mathfrak{S}_{n_0,r}$, there is a natural way to extend this to $\mathfrak{S}_{n,r}$ for any $n$. We describe this process below. 

\begin{definition} \label{def:SymmetricExtensions}
    Fix $n_0\ge 2$. Let $X$ be an order-invariant statistic on $\mathfrak{S}_{n_0,r}$ induced by $\mathcal{C}$. The \emph{$r$-colored order-invariant extension} of $X$ (or $\mathcal{C}$) are the statistics  $X_n=\sum_{(K,\kappa)\in \mathcal{C}_n} I_{(K,\kappa)}$ on $\mathfrak{S}_{n,r}$ with $\mathcal{C}_n$ defined as follows:
    \begin{itemize}
        \item If $n\leq n_0$, then $\mathcal{C}_n$ contains all $(K,\kappa)\in \mathcal{C}$ with support contained in $[n]$.
        \item If $n\geq n_0$, then $\mathcal{C}_n$ is the set of all $f(K,\kappa)$ where $(K,\kappa)\in \mathcal{C}$ and  $f:[n_0]\to [n]$ is order-preserving.
    \end{itemize}
    Observe that by construction, each $X_n$ is order-invariant. We emphasize here that $r$ is kept constant.
\end{definition}

As noted in the Introduction, many statistics on   $\mathfrak{S}_n$ can be constructed in this manner, including permutation pattern statistics. The statistic $\inv$ on $B_n$ is another instance of an order-invariant extension, as shown in the next example. Other examples on general colored permutation groups include the number of excedances or fixed points, as defined in \cite{CM2012}. 

\begin{ex}\label{ex:inv_symmetric}
    Let $X$ be the inversion statistic $\inv$ on $B_4$, 
    which can be expressed as 
    \[X=\sum_{ \substack{i,j \in [4]
     \\ i<j}} \, \sum_{\substack{k,\ell\in [\pm 4] \\ k<\ell}} \, I_{\{(i,\ell),(j,k)\}}.\]
    Let $\mathcal{C}_{4}$ denote the partial colored permutations inducing this statistic. Applying Definition~\ref{def:SymmetricExtensions}, observe that $\{X_n\}_{n\geq 1}$ are the inversion statistics on $\{B_n\}_{n\geq 1}$.
\end{ex}

In general, order-invariant extensions satisfy the following polynomial property. 

\begin{lemma}\label{lemma:Polynomial}
Fix $r\ge 1$, and let $(X_{n})_{n\geq 1}$ be the $r$-colored order-invariant extension of an order-invariant set of partial colored permutations $\mathcal{C}$. Suppose $\mathcal{C}$ consists of partial colored permutations of size $m$. Then there exists a polynomial $p(n)$ of degree at most $m$ with the following property: for any $n\geq 1$ and conjugacy class $C_{\lambda_n}$ of $\mathfrak{S}_{n,r}$ with no cycles of length $1,2,\ldots,m$, we have $p(n) = \E_{\rpart_{n}}[X_{n}]$.
\end{lemma}

\begin{proof}
    Letting $\mathcal{C}_n$ denote the set of partial colored permutations that induce $X_n$, partition $\mathcal{C}_n = \sqcup_{s=m+1}^{2m} \mathcal{C}_{n,s}$ where each $\mathcal{C}_{n,s}$ contains the partial colored permutations $(K,\kappa)\in \mathcal{C}_n$ with $|\text{supp}(K, \kappa)|=s$. Using this, we express
    \[X_n=\sum_{s=m+1}^{2m}\sum_{(K,\kappa)\in \mathcal{C}_{n,s}} I_{(K,\kappa)}.\]
    Defining $\mathcal{A}_{n,s}\subseteq \mathcal{C}_{n,s}$ to be the subset of acyclic partial colored permutations, Theorem \ref{thm:IndependenceHigherMoments} implies
    \begin{equation}\label{eq:orderinvariantmean}
        \E_{\rpart_n}[X_n]=\sum_{s=m+1}^{2m} \frac{|\mathcal{A}_{n,s}|}{(n-1)(n-2)\dots (n-m)}.
    \end{equation}

    We now turn our attention to $|\mathcal{A}_{n,s}|$. When $n<s$, the construction of $(X_n)_{n\geq 1}$ implies that $\mathcal{A}_{n,s}=\emptyset$, so $|\mathcal{A}_{n,s}|=0$. When $n\geq s$, we can express 
    \[\mathcal{A}_{n,s}=\bigsqcup_{\substack{f:[s]\to [n] \\ \text{ order-preserving}}} \{f(K,\kappa)\mid (K,\kappa)\in \mathcal{A}_{s,s} \}.\]
    Consequently, its cardinality is given by 
    \[|\mathcal{A}_{n,s}|={n\choose s}|\mathcal{A}_{s,s}|.\]
    In this case, each summand in \eqref{eq:orderinvariantmean} can be expressed as
    \begin{equation}\label{eq:orderinvariantmean2}
        \begin{split}
            \frac{|\mathcal{A}_{n,s}|}{(n-1)(n-2)\dots (n-m)r^m}  =  \frac{|\mathcal{A}_{s,s}|}{s!r^m} \cdot n(n-m-1)\cdots (n-s+1).
        \end{split}
    \end{equation}
    Observe that when $n<s$ is substituted, the right side of \eqref{eq:orderinvariantmean2} vanishes. Consequently, \eqref{eq:orderinvariantmean2} actually holds for all $n$. Combining Equations~\eqref{eq:orderinvariantmean} and \eqref{eq:orderinvariantmean2}, we find that 
    \begin{equation}\label{eq:orderinvariantmean3}
        \E_{\rpart_n}[X_n]=\sum_{s=m+1}^{2m} \frac{|\mathcal{A}_{s,s}|}{s!r^m} \cdot n(n-m-1)\cdots (n-s+1).
    \end{equation}
    This holds for all $\rpart_n$ without cycles of length $1,2,\ldots,m$, so this is the claimed polynomial $p(n)$. Note that each summand in \eqref{eq:orderinvariantmean3} has degree at most $s-m\leq m$, so $p(n)$ has degree at most $m$. 
\end{proof}

\begin{remark}
    In the proof of Lemma~\ref{lemma:Polynomial}, one can calculate $|\mathcal{A}_{s,s}|$ for all $s$ to obtain an explicit formula for $p(n)$. Alternatively, one can use the degree of $p(n)$ with polynomial interpolation to obtain an explicit formula. The practicality of these methods depends on the order-invariant extension under consideration.
\end{remark}

Our result Theorem~\ref{thm:Polynomiality} holds for arbitrary moments, while  Lemma~\ref{lemma:Polynomial} only includes the first moment. To obtain the result for any moment, we require additional properties involving the general class of order-invariant extensions. We describe these now.

\begin{definition}
    Define $\mathcal{I}_{r}$ to be the (real) vector space generated by $r$-colored order-invariant extensions with addition and scalar multiplication defined component-wise. In other words,
    \begin{itemize}
        \item for any order-invariant extensions $(X_n)_{n \geq 1}$ and $(Y_n)_{n\geq 1}$, define $(X_n)_{n \geq 1}+(Y_n)_{n\geq 1}=(X_n+Y_n)_{n\geq 1},$ and
        \item for any order-invariant extension $(X_n)_{n \geq 1}$ and constant $c\in \mathbb{R}$, define $c\cdot (X_n)_{n\geq 1}=(c\cdot X_n)_{n\geq 1}$.
    \end{itemize}
    Furthermore, define $\mathcal{I}_{r,m}$ to be the subspace generated by order-invariant extensions of any $\mathcal{C}$ consisting of partial colored permutations with size at most $m$.
\end{definition}

\begin{lemma}\label{lem:algebra}
    Let $(X_n)_{n\geq 1}$ and $(Y_n)_{n\geq 1}$ be the order-invariant extensions of some order-invariant statistics. Then $(X_nY_n)_{n\geq 1}\in \mathcal{I}_r$. In particular, if $(X_n)_{n\geq 1}\in \mathcal{I}_{r,m_1}$ and $(Y_n)_{n\geq 1}\in \mathcal{I}_{r,m_2}$, then $(X_nY_n)_{n\geq 1}\in \mathcal{I}_{r,m_1+m_2}$.
\end{lemma}

\begin{proof}
    Assume $(X_n)_{n\geq 1}$ is the order-invariant extension of some $\mathcal{C}_{1}$ and $(Y_n)_{n\geq 1}$ is the order-invariant extension of some $\mathcal{C}_{2}$. We let $m_1$ and $m_2$ be the sizes of the partial colored permutations in $\mathcal{C}_1$ and $\mathcal{C}_2$, respectively. Additionally, let $\mathcal{C}_{n,1}$ and $\mathcal{C}_{n,2}$ be the partial colored permutations that induce $X_n$ and $Y_n$, respectively. We can express the product
    \[X_nY_n=\sum_{(K_1,\kappa_1)\in \mathcal{C}_1}\sum_{(K_2,\kappa_2)\in \mathcal{C}_2} I_{(K_1,\kappa_1)}I_{(K_2,\kappa_2)}.\]
    By Lemma~\ref{lem:HRLem4.2}, it suffices to consider when $(K_1,\kappa_1)$ and $(K_2,\kappa_2)$ are compatible. Consequently, define 
    \[\mathcal{C}_n=\{ ((K_1,\kappa_1),(K_2,\kappa_2)) \mid  (K_1,\kappa_1)\in \mathcal{C}_{n,1} \text{ and } (K_2,\kappa_2)\in \mathcal{C}_{n,2} \text{ are compatible}\}.\]
    Letting $\mathcal{C}_{n,t,s}\subseteq \mathcal{C}_n$ be the subset consisting of pairs such that $(K_1\cup K_2,\kappa_1\cup \kappa_2)$ is a partial colored permutation on $\mathfrak{S}_{n,r}$ with size $t$ and support of size $s$, this results in a partition of $\mathcal{C}_n$ into
    \[\mathcal{C}_n=\bigsqcup_{t=\max(m_1,m_2)}^{m_1+m_2} \bigsqcup_{s=t+1}^{2t} \mathcal{C}_{n,t,s}.\]
    By defining the statistic
    \begin{equation}\label{eq:component}
        Z_{n,t,s}=\sum_{((K_1,\kappa_1),(K_2,\kappa_2))\in \mathcal{C}_{n,t,s}} I_{(K_1\cup K_2,\kappa_1\cup \kappa_2)},
    \end{equation}
    we can decompose 
    \begin{equation}\label{eq:XYsummation}
        (X_nY_n)_{n\geq 1} =\sum_{t=\max(m_1,m_2)}^{m_1+m_2} \sum_{s=t+1}^{2t} (Z_{n,t,s})_{n\geq 1}
    \end{equation}
    where the summation on the right-hand side is component-wise. It now suffices to show that $(Z_{n,t,s})_{n\geq 1}\in \mathcal{I}_{r,m_1+m_2}$. 

    To show this, we decompose
    \begin{equation}\label{eq:Cnts}
        \mathcal{C}_{n,t,s}=\bigsqcup_{\substack{f:[s]\to [n] \\ \text{ order-preserving}}} \{(f(K_1,\kappa_1),f(K_2,\kappa_2))\mid ((K_1,\kappa_1),(K_2,\kappa_2))\in \mathcal{C}_{s,t,s}\}.
    \end{equation}
    Note that when $n<s$, the right hand side should be interpreted as the empty union, which results in $\emptyset$. Focusing on the set $\mathcal{C}_{s,t,s}$, observe that the support of  $(K_1\cup K_2,\kappa_1\cup \kappa_2)$ for each $((K_1,\kappa_1),(K_2,\kappa_2))\in \mathcal{C}_{s,t,s}$ is all of $[s]$. Consequently, this single element forms an order-invariant set of partial colored permutations on $\mathfrak{S}_{n,r}$, and the statistic $I_{(K_1\cup K_2,\kappa_1\cup \kappa_2)}$ on $\mathfrak{S}_{s,r}$ is order-invariant. Denoting its order-invariant extension as $(I_{n,(K_1\cup K_2,\kappa_1\cup \kappa_2)})_{n\geq 1}$, we can combine this with Equations \eqref{eq:Cnts} and \eqref{eq:component} to find
    \[(Z_{n,t,s})_{n\geq 1} =\sum_{((K_1,\kappa_1),(K_2,\kappa_2))\in \mathcal{C}_{s,t,s}} (I_{n,(K_1\cup K_2,\kappa_1\cup \kappa_2)})_{n\geq 1}.\]
    Thus, we see that $(Z_{n,t,s})_{n\geq 1}\in \mathcal{I}_{r,t}\subseteq \mathcal{I}_{r,m_1+m_2}$. The result $(X_nY_n)_{n\geq 1}\in \mathcal{I}_{r,m_1+m_2}$ now follows from \eqref{eq:XYsummation}.
\end{proof}

\begin{corollary}\label{cor:algebra}
    The vector space $\mathcal{I}_r$ has the structure of an $\mathbb{R}$-algebra with multiplication defined by $(X_n)_{n\geq 1}\cdot (Y_n)_{n\geq 1}=(X_nY_n)_{n\geq 1}$. Furthermore, $\mathcal{I}_{r,m_1}\cdot \mathcal{I}_{r,m_2}\subseteq \mathcal{I}_{r,m_1+m_2}$.
\end{corollary}

Using the preceding results, we can now establish a more general case of Theorem~\ref{thm:Polynomiality}.

\begin{theorem}\label{thm:HigherMomentPolynomial}
Fix $r\ge 1$, and let $(X_{n})_{n\geq 1}\in \mathcal{I}_{r,m}$. For each $k\geq 1$, there exists a polynomial $p(n)$ of degree at most $mk$ with the following property: for any $n\geq 1$ and conjugacy class $C_{\lambda_n}$ of $\mathfrak{S}_{n,r}$ with no cycles of length $1,2,\ldots,mk$, we have $p(n) = \E_{\rpart_{n}}[X_{n}^k]$.
\end{theorem}

\begin{proof}
    By Corollary~\ref{cor:algebra}, we see that $(X_n^k)_{n\geq 1}$ can be expressed as a linear combination of order-invariant extensions
    \[(X_n^k)_{n\geq 1}=c_1\cdot (Y_{n,1})_{n\geq 1} + c_2\cdot (Y_{n,2})_{n\geq 1}+\ldots + c_j \cdot (Y_{n,j})_{n\geq 1}\]
    where each $(Y_{n,i})_{n\geq 1}\in \mathcal{I}_{n,mk}$. Consequently, we have that for all $n$,
    \[\E_{\rpart_n}[X_n^k]=c_1 \cdot \E_{\rpart_n}[Y_{n,1}]+c_2 \cdot \E_{\rpart_n}[Y_{n,2}]+\ldots +c_j \cdot \E_{\rpart_n}[Y_{n,j}].\]
    Applying Lemma~\ref{lemma:Polynomial} to each term on the right side implies that this is given by a polynomial in $n$ with degree at most $mk$ whenever $\rpart_n$ does not contain cycles of length $1,2,\ldots,mk$. 
\end{proof}

\begin{remark}
The proofs in this section bear some similarities with the methods used by Chern, Diaconis, Kane, and Rhoades \cite{CDKR} in their study of statistics on set partitions. For example, see the proof of \cite[Lemma 7]{CDKR}, which also analyzes terms in a summation and shows that certain terms must be polynomials in the appropriate variables. 
\end{remark}

Note that one can potentially show the polynomiality property from Theorem~\ref{thm:HigherMomentPolynomial} for other statistics that are not order-invariant extensions. The key requirement is that the weights for the various $I_K$ behave in a way that allows us to divide by the denominators resulting from applying Lemma~\ref{lem:HigherMomentsIndependence}.

\section{Inversion and descent statistics in $B_n$}\label{Bnstatistics}

In this section, we examine the inversion and descent statistics in the hyperoctahedral group, which have degree at most $2$. Our results in the previous section imply that we expect the mean number of descents and inversions on a conjugacy class $C_{\lambda,\mu}$ to depend on $n$ and the number of cycles of length $1$ or $2$ in $\lambda$ and $\mu$. However, in deriving our explicit formulas, we will find that a much stronger statement is true: the means of these statistics only depend on $n$ and differences of the form $m_1(\lambda)^k-m_1(\mu)^k$ for $k\in \mathbb{N}$. In particular, we will see that the converse of Theorem \ref{thm:IndependenceHigherMoments} does not hold. 

Throughout this section, we will fix a hyperoctahedral group $B_n$. We will also fix an arbitrary conjugacy class $C_{\lambda,\mu}$, which is indexed by a bi-partition $(\lambda,\mu)$ of $n$.

\subsection{Inversion indicator functions}

In this subsection, we focus on the indicator function $\inv_{ij}$ for $(i,j)$ being an inversion, where $i<j$ are elements in $[n]$. Each $\inv_{ij}$ can be expressed using partial colored permutations of size $2$, and $\inv$ can be expressed in terms of $\inv_{ij}$ as follows:
\begin{equation}\label{eq:size2decomposition}
    \inv_{ij}=\sum_{\substack{k,\ell\in [\pm n]\\k<\ell} }I_{\{(i,\ell),(j,k)\}} \quad \text{ and }\quad \inv=\sum_{\substack{i,j\in [n]\\i<j} } \inv_{ij}.
\end{equation}
To determine $\E_{\lambda,\mu}[\inv]$, it suffices to compute $\E_{\lambda,\mu}[\inv_{ij}]$ over appropriate $i$ and $j$.

We start with some lemmas that will be used repeatedly throughout our work. We first make an observation about symmetry in the set of conjugacy classes of $B_n$ and then establish some results related to the size of various subsets in $C_{\lambda,\mu}$. Recall from Definition~\ref{def:doublezee} that $z_{\lambda, \mu}=2^{\ell(\lambda)} z_\lambda 2^{\ell(\mu)} z_\mu$.

\begin{lemma}\label{lem:symm-Bnconjugacyclasses} There is a bijection $\phi:C_{\lambda,\mu}\rightarrow C_{\mu, \lambda}$, which switches even and odd cycles. More generally, in the colored permutation group $\Sym_{n,r}$, let $C_{\rpart}$ be the conjugacy class indexed by the $r$-partition $\rpart=(\lambda^{0}, \lambda^{1}, \ldots, \lambda^{r-1})$. For any pair $i, j$, let $\rpart(i,j) $ be the $r$-partition obtained from $\rpart$ by switching the $i$th and $j$th partitions in $\rpart$. Then there is a bijection  $\phi_{i,j}:C_{\rpart}\rightarrow  C_{\rpart(i,j)}$.
\end{lemma}
\begin{proof} 
One such bijection $\phi$ is obtained by simply negating  the smallest element (in absolute value) in each cycle. In Example~\ref{ex:basic1}, we have 
$(2,7,6,3,-1)(-5,8,-4)\mapsto (2,7,6,3,\mathbf{+1})(-5,8,\mathbf{+4})$, and in Example~\ref{ex:basic2}, 
$(2,7,-6,-1)(3)(-4)(8,-5)\mapsto (2,7,-6,\mathbf{+1})(\mathbf{-3})(\mathbf{+4})(8,\mathbf{+5})$. Clearly $\phi$ changes the parity of each cycle by definition, and $\phi^{-1}=\phi$.

For $\Sym_{n,r}$, given an $r$-partition $\rpart$ and fixed colors $i$ and $j$, one can more generally define the map $\phi_{i,j}$ that adds $j-i$ (modulo $r$) to the color of the smallest element, in $[n]$, in each cycle of color $i$, and that adds $i-j$ (modulo $r$) to the color of the smallest element, in $[n]$, in each cycle of color $j$.  This interchanges cycles of color $i$ with cycles of color $j$. Clearly the inverse map of $\phi_{i,j}$ is $\phi_{j,i}$, finishing the proof.
\end{proof}

\begin{lemma}\label{1cycles}
For fixed, distinct $i,j\in [n]$, the following hold:
\begin{enumerate}
    \item $\pr_{\lambda,\mu}[\omega(i)=i]=\frac{{m_1(\lambda)}}{n},$
    \item $\pr_{\lambda,\mu}[\omega(i)=-i]=\frac{{m_1(\mu)}}{n},$
    \item $\pr_{\lambda,\mu}[\omega(i)=i \wedge \omega(j)=j]=\frac{{m_1(\lambda)}({m_1(\lambda)}-1)}{n(n-1)},$
    \item $\pr_{\lambda,\mu}[\omega(i)=-i \wedge \omega(j)=j]=\frac{{m_1(\lambda)} {m_1(\mu)}}{n(n-1)},$ and
    \item $\pr_{\lambda,\mu}[\omega(i)=-i \wedge \omega(j)=-j]=\frac{{m_1(\mu)}({m_1(\mu)}-1)}{n(n-1)}.$
\end{enumerate}
\end{lemma}

\begin{proof}
Observe that the five claims respectively hold when ${m_1(\lambda)}=0$, ${m_1(\mu)}=0$, ${m_1(\lambda)}\leq 1$, ${m_1(\lambda)}=0$ or ${m_1(\mu)}=0$, and ${m_1(\mu)}\leq 1$. We consider the remaining cases. Observe that $\omega(i)=i$ fixes an even cycle of length $1$ and $\omega(i)=-i$ fixes an odd cycle of length $1$. Letting $\lambda'$ and  $\mu'$ be the partitions resulting from deleting an instance of $1$ in $\lambda$ and $\mu$, respectively, we can apply Lemma \ref{lem:classsize} to conclude that
\[\pr_{\lambda,\mu}[\omega(i)=i]=\frac{|C_{\lambda',\mu}|}{|C_{\lambda,\mu|}}=\frac{{m_1(\lambda)}}{n} \quad \text{ and } \quad \pr_{\lambda,\mu}[\omega(i)=-i]=\frac{|C_{\lambda,\mu'}|}{|C_{\lambda,\mu}|}=\frac{{m_1(\mu)}}{n}.\]

Similarly, the remaining three cases fix two even cycles of length $1$, an even and an odd cycle of length $1$, and two odd cycles of length $1$ respectively. Letting $\lambda''$ and $\mu''$ denote the partitions resulting from deleting an instance of $2$  from $\lambda$ and $\mu$, respectively, we similarly find
\[\pr_{\lambda,\mu}[\omega(i)=i \wedge \omega(j)=j]=\frac{|C_{\lambda'',\mu}|}{|C_{\lambda,\mu}|}=\frac{{m_1(\lambda)}({m_1(\lambda)}-1)}{n(n-1)},\]
\[\pr_{\lambda,\mu}[\omega(i)=-i \wedge \omega(j)=j]=\frac{|C_{\lambda',\mu'}|}{|C_{\lambda,\mu}|}=\frac{{m_1(\lambda)} {m_1(\mu)}}{n(n-1)},\]
\[\pr_{\lambda,\mu}[\omega(i)=-i \wedge \omega(j)=-j]=\frac{|C_{\lambda,\mu''}|}{|C_{\lambda,\mu}|}=\frac{{m_1(\mu)}({m_1(\mu)}-1)}{n(n-1)}.\qedhere\]
\end{proof}

 We now consider $\inv_{ij}$ where $i,j\in [n]$ and $i<j$. We consider the following partition of $C_{\lambda,\mu}$:
\begin{equation}\label{conj-class-partition}
\begin{split}
\Omega_{1+}^{ij}&=\{\omega \in C_{\lambda,\mu}\mid \omega(i)=i,\omega(j)=j\},\\
\Omega_{1-}^{ij}&=\{\omega \in C_{\lambda,\mu}\mid \omega(i)=-i,\omega(j)=-j\},\\
\Omega_{2+}^{ij}&=\{\omega \in C_{\lambda,\mu}\mid \omega(i)=i,\omega(j)\neq \pm j\}\cup \{\omega\in C_{\lambda,\mu}\mid \omega(j)=j,\omega(i)\neq \pm i\},\\
\Omega_{2-}^{ij}&=\{\omega \in C_{\lambda,\mu}\mid \omega(i)=-i,\omega(j)\neq \pm j\}\cup \{\omega\in C_{\lambda,\mu}\mid \omega(j)=-j,\omega(i)\neq \pm i\},\\
\Omega_{3\phantom-}^{ij} & = C_{\lambda,\mu}\setminus (\Omega_{1+}^{ij} \cup \Omega_{1-}^{ij} \cup \Omega_{2+}^{ij} \cup \Omega_{2-}^{ij}).
\end{split}
\end{equation}
We start by considering the last set, which decomposes into subsets where the proportion of elements with an inversion at $(i,j)$ is exactly $1/2$.

\begin{lemma}\label{omega3}
Fix $i<j$ with $i,j\in [n]$. Half of the elements in $\Omega_3^{ij}$ have the inversion $(i,j)$. Consequently, 
\[ \pr_{\lambda,\mu}[\inv_{ij}=1\mid \Omega_3^{ij}]=\frac{1}{2} \quad \text{ and } \quad \pr_{\lambda,\mu}[\{\inv_{ij}=1\}\cap \Omega_3^{ij}]=\frac{1}{2}\pr_{\lambda,\mu}[\Omega_3^{ij}].\]
\end{lemma}
\begin{proof}
We prove this by partitioning  $\Omega_3^{ij}\subseteq C_{\lambda,\mu}$ into subsets and defining involutions on each subset that interchange elements with or without the inversion $(i,j)$. Our partition is based on whether or not $i$ and $j$ are mapped to $\{\pm i,\pm j\}$, and all of our involutions involve conjugation, so we collect this information into Table \ref{tab:omega3} below.
\begin{table}[h!]
    \centering
    \begin{tabular}{|c|c|c|}
    \hline 
       set number & definition & conjugation by \\
        \hline 
        1 & $\{\omega\mid \omega(i)=i,\omega(j)=-j\}\cup \{\omega\mid \omega(i)=-i,\omega(j)=j\}$ & $(i,j)$ \\ \hline 
        2 & $\{\omega\mid \omega(i)=j,\omega(j)=i\}\cup \{\omega\mid \omega(i)=-j,\omega(j)=-i\}$ & $(-i)$ or $(-j)$ \\ \hline 
        3 & $\{\omega\mid \omega(i)=j,\omega(j)=-i\}\cup \{\omega\mid \omega(i)=-j,\omega(j)=i\}$ & $(-i)$ or $(-j)$ \\ \hline 
        4 & $\{\omega\mid \omega(i)=\pm j,\omega(j)\neq \pm i\}$ & $(-j)$ \\ \hline 
        5 & $\{\omega\mid \omega(i)\neq \pm j, \omega(j)=\pm i\}$ & $(-i)$ \\ \hline 
        6 & $\{\omega\mid \omega(i),\omega(j)\notin \{\pm i,\pm j\}\}$ & $(i,j)$ \\ \hline 
    \end{tabular}
    \caption{A partition of $\Omega_3^{ij}$ and involutions that interchange elements with or without the inversion $(i,j)$.}
    \label{tab:omega3}
\end{table}

It is clear that every element in $\Omega_3$ appears in exactly one of these sets. Considering each of these, we make the following conclusions on the respective sets.
\begin{enumerate}
    \item Conjugation by $(i,j)$ interchanges $\{\omega\mid \omega(i)=i,\omega(j)=-j\}$ and $\{\omega\mid \omega(i)=-i,\omega(j)=j\}$. All elements in the first subset have the inversion $(i,j)$, and all elements in the second set do not.
    \item Conjugation by $(-i)$ or $(-j)$ interchanges $\{\omega\mid \omega(i)=j,\omega(j)=i\}$ and $\{\omega\mid \omega(i)=-j,\omega(j)=-i\}$. All elements in the first subset have the inversion $(i,j)$, and all elements in the second set do not. 
    \item Conjugation by $(-i)$ or $(-j)$ interchanges $\{\omega\mid \omega(i)=j,\omega(j)=-i\}$ and $\{\omega\mid \omega(i)=-j,\omega(j)=i\}$. All elements in the first subset have the inversion $(i,j)$, and all elements in the second set do not. 
    \item If we fix $k$, conjugation by $(-j)$ interchanges $\{\omega \mid \omega(i)=j,\omega(j)=k\}$ and $\{\omega\mid \omega(i)=-j,\omega(j)=-k\}$. Exactly one of $j>k$ or $j<k$ is true, and the latter is equivalent to $-j>-k$.
    \item Again, we fix $k$ and conjugate by $(-i)$ to interchange $\{\omega\mid \omega(i)=k,\omega(j)=i\}$ and $\{\omega\mid \omega(i)=-k,\omega(j)=-i\}$. Exactly one of $k>i$ or $k<i$ is true, and the latter is equivalent to $-k>-i$.
    \item On this set, conjugation by $(i,j)$ simply interchanges the images of $i$ and $j$.
\end{enumerate}
Hence, half of the elements in $\Omega_3^{ij}$ have the inversion $(i,j)$.
\end{proof}

The remaining four subsets in (\ref{conj-class-partition}) take some additional care. We establish the corresponding results for them in the next lemma. 

\begin{lemma}\label{omega12}
Fix $i<j$ with $i,j\in [n]$. Then Table \ref{tab:omega12} holds.
\begin{center}
\begin{table}[h]
    \begin{tabular}{|c|c|c|}
        \hline 
         set $\Omega$ & $\pr_{\lambda,\mu}[\Omega]$ & $\pr_{\lambda,\mu}[\inv_{ij}=1\mid\Omega]$\\
        \hline 
        
         $\Omega_{1+}^{ij}$ & $\frac{{m_1(\lambda)}({m_1(\lambda)}-1)}{n(n-1)}$ & $0$ \\ \hline 
         $\Omega_{1-}^{ij}$ & $\frac{{m_1(\mu)}({m_1(\mu)}-1)}{n(n-1)}$ & $1$ \\ \hline 
        
        $\Omega_{2+}^{ij}$ & $2\cdot\left(\frac{{m_1(\lambda)}}{n}-\frac{{m_1(\lambda)}({m_1(\lambda)}-1)+{m_1(\lambda)}{m_1(\mu)}}{n(n-1)}\right)$ &  
        \makecell{$\frac{1}{2}-\frac{j-i-1}{4(n-2)}$ }
        \\  \hline
         $\Omega_{2-}^{ij}$ & $2\cdot\left(\frac{{m_1(\mu)}}{n}-\frac{{m_1(\mu)}({m_1(\mu)}-1)+{m_1(\lambda)}{m_1(\mu)}}{n(n-1)}\right)$ & \makecell{$\frac{1}{2}+\frac{j-i-1}{4(n-2)}$}\\  \hline 
    \end{tabular}
    \caption{Sizes and conditional probabilities for the remaining sets in $C_{\lambda,\mu}$.}
    \label{tab:omega12}
\end{table}
\end{center}
\end{lemma}

\begin{proof}
    The results for $\Omega_{1+}^{ij}$ and $\Omega_{1-}^{ij}$ follow immediately from Lemma \ref{1cycles} and the definition of $\inv_{ij}$. For $\Omega_{2+}^{ij}$, Lemma \ref{1cycles} implies
    \begin{equation*}
        \begin{split}
            \pr_{\lambda,\mu}[\omega(i)=i,\omega(j)\neq \pm j] 
            & =\pr_{\lambda,\mu}[\omega(i)=i]-\pr_{\lambda,\mu}[\omega(i)=i,\omega(j)=\pm j] \\
            & = \left(\frac{{m_1(\lambda)}}{n}-\frac{{m_1(\lambda)}({m_1(\lambda)}-1)+{m_1(\lambda)}{m_1(\mu)}}{n(n-1)}\right).
        \end{split}
    \end{equation*}
    The same result holds when $i$ and $j$ are interchanged, so $\pr_{\lambda,\mu}[\Omega_{2+}^{ij}]$ is twice the above quantity. The corresponding result for $\Omega_{2-}^{ij}$ follows by combining this with Lemma~\ref{lem:symm-Bnconjugacyclasses}.

    Using conjugation by a cycle containing the elements in $[n]\setminus \{i,j\}$ implies that $\{w\in \Omega_{2+}^{ij}\mid \omega(i)=i,\omega(j)=k\}$ for $k\in [n]\setminus \{i,j\}$ all have the same size. Conjugating by $(-j)$ then implies $\{w\in \Omega_{2+}^{ij}\mid \omega(i)=i,\omega(j)=k\}$ and $\{w\in \Omega_{2+}^{ij}\mid \omega(i)=i,\omega(j)=-k\}$ have the same size. Conjugation by $(ij)$ then implies $\{w\in \Omega_{2+}^{ij}\mid \omega(i)=i,\omega(j)=k\}$ has the same size as $\{w\in \Omega_{2+}^{ij}\mid \omega(i)=k,\omega(j)=j\}$. Hence, $\Omega_{2+}^{ij}$ decomposes into $4(n-2)$ sets of the same size based on the images of $i$ and $j$. Observe that $\inv_{ij}$ is constant on each of these sets.
    
    We now consider some casework on $\Omega_{2+}^{ij}$ to establish the explicit value in Table \ref{tab:omega12} for $\pr_{\lambda,\mu}[\inv_{ij}=1\mid \Omega_{2+}^{ij}]$. If $\omega(i)=i$, then the possibilities for $\omega(j)$ that result in an inversion are elements in $\{-n,\ldots,i-1\}\setminus \{0,-i,-j\}$ for a total of $n+i-3$ possibilities. If $\omega(j)=j$, then there are $n-j$ possibilities given by $\{j+1,\ldots,n\}$. Combined, we conclude that $\inv_{ij}$ takes value $1$ on a total of $2n-j+i-3$ of the $4(n-2)$ sets from the preceding paragraph, and it takes value 0 on all remaining sets. Hence,
    \[\pr_{\lambda,\mu}[\inv_{ij}=1\mid \Omega_{2+}^{ij}]=\frac{2n-j+i-3}{4(n-2)}
    =\frac{1}{2}-\frac{j-i-1}{4(n-2)}.\]
    The corresponding result for $\Omega_{2-}^{ij}$ follows by combining this with Lemma~\ref{lem:symm-Bnconjugacyclasses}.
\end{proof}

We can now derive an explicit formula for $\E_{\lambda,\mu}[\inv_{ij}]$. The expression below will appear frequently in our calculations.

\begin{definition}\label{def:Delta-biptn} For a bi-partition $(\lambda,\mu)$ and a nonnegative integer $k$, define a function $\Delta^k$ by
\[\Delta^k(\lambda, \mu):=m_1(\lambda)^k-m_1(\mu)^k.\]  
\end{definition}

An immediate corollary of  Lemma~\ref{lem:symm-Bnconjugacyclasses} is the following vanishing condition for the  function  $\Delta^k(\lambda,\mu)$:
\begin{equation}\label{eqn:vanishing-condition}
\sum_{(\lambda,\mu)\vdash n} z^{-1}_{\lambda,\mu} \Delta^k(\lambda,\mu)=0.
\end{equation}

\begin{proposition}\label{ij-indicator}
    For any $i<j$ with $i,j\in [n]$,  
 \begin{equation*}   
 \E_{\lambda,\mu}[\inv_{ij}]=\frac{1}{2}\left(1 +
\frac{\Delta^1(\lambda,\mu)-\Delta^2(\lambda, \mu)}{n(n-1)} \right)
-\frac{j-i-1}{2(n-2)}\left( 
\frac{\Delta^1(\lambda,\mu)}{n-1}- 
\frac{\Delta^2(\lambda,\mu)}{n(n-1)}
\right).
\end{equation*}
\end{proposition}

\begin{proof}
    Using Lemmas \ref{omega3} and \ref{omega12}, we find that
    \begin{enumerate}
        \item $\pr_{\lambda,\mu}[\Omega_{1+}^{ij}\cap\{\inv_{ij}=1\}]=\frac{1}{2}\pr_{\lambda,\mu}[\Omega_{1+}^{ij}]-\frac{{m_1(\lambda)}({m_1(\lambda)}-1)}{2n(n-1)}$,
        \item $\pr_{\lambda,\mu}[\Omega_{1-}^{ij}\cap \{\inv_{ij}=1\}]=\frac{1}{2}\pr_{\lambda,\mu}[\Omega_{1-}^{ij}]+\frac{{m_1(\mu)}({m_1(\mu)}-1)}{2n(n-1)}$,
        \item $\pr_{\lambda,\mu}[\Omega_{2+}^{ij}\cap \{\inv_{ij}=1\}]=\frac{1}{2}\pr_{\lambda,\mu}[\Omega_{2+}^{ij}]-\left(\frac{{m_1(\lambda)}}{n}-\frac{{m_1(\lambda)}({m_1(\lambda)}-1)+{m_1(\lambda)}{m_1(\mu)}}{n(n-1)}\right)\cdot \frac{j-i-1}{2(n-2)}$, 
        \item $\pr_{\lambda,\mu}[\Omega_{2-}^{ij}\cap \{\inv_{ij}=1\}]=\frac{1}{2}\pr_{\lambda,\mu}[\Omega_{2-}^{ij}]+\left(\frac{{m_1(\mu)}}{n}-\frac{{m_1(\mu)}({m_1(\mu)}-1)+{m_1(\lambda)}{m_1(\mu)}}{n(n-1)}\right)\cdot \frac{j-i-1}{2(n-2)}$, and
        \item $\pr_{\lambda,\mu}[\Omega_3^{ij}\cap \{\inv_{ij}=1\}]=\frac{1}{2}\pr_{\lambda,\mu}[\Omega_3^{ij}]$.
    \end{enumerate}
    Adding the above terms, we obtain
    \begin{equation*} 
    \begin{split}
        &\frac{1}{2}-\frac{m_1(\lambda)^2-{m_1(\lambda)}}{2n(n-1)}+\frac{{m_1(\mu)}^2-{m_1(\mu)}}{2n(n-1)}
    -\frac{m_1(\lambda)}{n}\cdot \frac{j-i-1}{2(n-2)}+\frac{m_1(\mu)}{n}\cdot \frac{j-i-1}{2(n-2)}\\
    &+\frac{m_1(\lambda)^2-m_1(\lambda)}{n(n-1)}\cdot \frac{j-i-1}{2(n-2)}
    -\frac{m_1(\mu)^2-m_1(\mu)}{n(n-1)}\cdot \frac{j-i-1}{2(n-2)}.
    \end{split}
    \end{equation*}
    Regrouping terms, we have
    \begin{equation*} 
    \begin{split}
    & \frac{1}{2}-\frac{m_1(\lambda)^2-{m_1(\mu)^2}}{2n(n-1)}+\frac{{m_1(\lambda)}-{m_1(\mu)}}{2n(n-1)} \\
    & 
    -\frac{m_1(\lambda)-m_1(\mu)}{n-1}\cdot \frac{j-i-1}{2(n-2)}+\frac{m_1(\lambda)^2-m_1(\mu)^2}{n(n-1)}\cdot \frac{j-i-1}{2(n-2)}.
    \end{split}
    \end{equation*}
    Applying the definition of $\Delta^k(\lambda,\mu)$ and factoring appropriately then establishes the result.
\end{proof}

\begin{corollary}
    For any $i,j\in [n]$ where $i<j$, we have
    \[\E_{\lambda,\mu}[\inv_{ij}]+\E_{\mu,\lambda}[\inv_{ij}]=1.\]
\end{corollary}

\begin{remark} We can also analyze indicator functions $\inv_{i j}$ where $i,j$ are not necessarily positive and obtain similar formulas. Since $\inv_{i,j}=\inv_{-i,-j}$, we only need to separately consider $i<0<j$. When $j\neq -i$, the appearance of $j-i-1$ in the table of Lemma~\ref{omega12} would be replaced with $j-i-3$. We will consider the case of $j=-i$ in Section~\ref{sec:4.2}.
\end{remark}

As we now have an explicit formulation of $\E_{\lambda,\mu}[\inv_{ij}]$ for $i,j\in [n]$ with $i<j$, we can derive explicit formulations for the means of any statistics that can be expressed as a weighted linear combination of these $\inv_{ij}$. This motivates our next definition.

\begin{definition}\label{def:GenInvStat} For any function $\mathrm{wt}_n :[ n]\times [n] \rightarrow \mathbb{R}, $ define the weighted  inversion statistic $X_n^{wt}$ to be \[X_n^{wt}:=\sum_{i,j\in [n], i<j}\mathrm{wt}_n(i,j)\inv_{ij}.\] Additionally, define 
\begin{equation}
    \label{eqn:alpha-beta}
\alpha(X_n^{wt}):=\sum_{{i,j\in [n]}}\wt_n(i,j), \quad
\beta(X_n^{wt}):=\sum_{{i,j\in [n],i<j}} (j-i-1) \wt_n(i,j).
\end{equation}
\end{definition}

We now apply our results on $\E_{\lambda,\mu}[\inv_{i,j}]$ to calculate $\E_{\lambda, \mu}[X_n^{wt}]$ for any weighted inversion statistic $X_n^{wt}$.

\begin{theorem}\label{thm:generalinversion}
Let $X_n^{wt}=\sum_{i,j\in [n]} \wt_n(i,j) \inv_{i,j}$ be a weighted inversion statistic.  Then $\E_{\lambda,\mu}[X_n^{wt}]$ depends only on $n,m_1(\lambda),$ and $m_1(\mu)$. Furthermore, it is given by the formula
\begin{align*}
\frac{\alpha(X_n^{wt})}{2}\left(1 +
\frac{\Delta^1(\lambda,\mu)-\Delta^2(\lambda, \mu)}{n(n-1)} \right)
-\frac{ 
\beta(X_n^{wt})}{2(n-2)}\left( 
\frac{\Delta^1(\lambda,\mu)}{n-1}- 
\frac{\Delta^2(\lambda,\mu)}{n(n-1)}
\right).
\end{align*} 
Moreover, the first moment on all of $B_n$ is 
\begin{equation}\label{eqn:gen-inv-first-moment-Bn}
\E_{B_n}[X_n^{wt}]=\frac{\alpha(X_n^{wt})}{2}.
\end{equation}
Additionally,
\begin{equation}\label{eqn:1stmoment-conjclass-wholegp}\E_{\lambda, \mu}[X_n^{\wt}]=\E_{B_n}[X_n^{wt}] +f_n(m_1(\lambda), m_1(\mu))
\end{equation}  
where, for fixed $n$,  $f_n(a_1, b_1)$ is a polynomial  of degree at most 2 in the two variables $a_1, b_1$  such that
\begin{equation}\label{eqn:vanishing-condition-fpoly}\sum_{(\lambda,\mu)\vdash n} z^{-1}_{\lambda,\mu} f_n(m_1(\lambda), m_1(\mu))=0. \end{equation}
\end{theorem}

\begin{proof}
Proposition~\ref{ij-indicator} implies
\[ \E_{\lambda, \mu}[\inv_{i,j}]=\frac{1}{2}\left(1 +
\frac{\Delta^1(\lambda,\mu)-\Delta^2(\lambda, \mu)}{n(n-1)} \right)
-\frac{j-i-1}{2(n-2)}\left( 
\frac{\Delta^1(\lambda,\mu)}{n-1}- 
\frac{\Delta^2(\lambda,\mu)}{n(n-1)}
\right).\]
The expression  for 
$\E_{\lambda, \mu}[X_n^{\wt}]$ follows by linearity since we have  
\[ \E_{\lambda, \mu}[X_n^{\wt}]
=\sum_{i,j\in [\pm n], i<j} \wt(i,j)\, \E_{\lambda, \mu}[\inv_{i,j}].\]
Because of  \eqref{eqn:vanishing-condition}, the  statement \eqref{eqn:gen-inv-first-moment-Bn}  about the first moment on all of $B_n$   now follows from the observation that 
for any statistic $X$ on $B_n$, we have 
\[\E_{B_n}[X]=\sum_{(\lambda,\mu)\vdash n} z^{-1}_{\lambda,\mu} \E_{\lambda,\mu}[X]. \qedhere\]
Setting $f_n(m_1(\lambda), m_1(\mu))=\E_{\lambda, \mu}[X_n^{\wt}]-\E_{B_n}[X_n^{wt}]$ then results in a polynomial satisfying~\eqref{eqn:1stmoment-conjclass-wholegp} and~\eqref{eqn:vanishing-condition-fpoly}. 
\end{proof}

The following corollary is clear.
\begin{corollary}\label{cor:stable-moments} 
 Let $X_n^{\wt}$ be a weighted inversion statistic on $B_n$. If $m_1(\lambda)=m_1(\mu)$,  then
 \[\E_{\lambda, \mu}[X_n^{\wt}]=\E_{B_n}[X_n^{\wt}].\]
 In particular, this holds when $C_{\lambda,\mu}$ contains no cycles of length $1$ or when $\lambda=\mu$.
\end{corollary}

\begin{remark}
The analogous analysis for weighted inversion statistics in $\mathfrak{S}_n$ appears in \cite[Proposition~6.2]{GRWCPermutationStatistics}. There the conclusion about the first moment on the entire group used character computations, see \cite[Lemma~6.1]{GRWCPermutationStatistics}. In contrast, \eqref{eqn:1stmoment-conjclass-wholegp} follows simply from the symmetry in the distribution of conjugacy classes observed in Lemma~\ref{lem:symm-Bnconjugacyclasses} and the resulting vanishing condition \eqref{eqn:vanishing-condition}.
\end{remark}

\subsection{Formulas for means on conjugacy classes of $B_n$}\label{sec:4.2}

We now apply our results from the previous section to establish explicit formulas for the descent and inversion statistics in conjugacy classes of the hyperoctahedral group. We consider a few auxiliary statistics needed to calculate $\des_B$ and $\inv_B$, and then specialising Theorem~\ref{thm:generalinversion} in suitable ways will allow us to establish these formulas. 

For any $i\in [n]$, define $\inv_{-i,i}$ to be the indicator function for a permutation $\omega$ satisfying $\omega(i)<\omega(-i)$, which can be viewed as the pair $(-i,i)$ being an inversion. Note that this is also equivalent to $\omega(i)<0$. We now establish a formula for the mean of $\inv_{-i,i}$ on any conjugacy class and then use it to derive a corresponding expression for $\des_B$.

\begin{proposition}\label{-ii-inv}
For any $i\in [n]$, we have that
\[\E_{\lambda,\mu}[\inv_{-i,i}]=\pr_{\lambda,\mu}[\inv_{-i,i}=1]=\frac{1}{2}-\frac{\Delta^1(\lambda,\mu)}{2n}.\]
\end{proposition}

\begin{proof}
    In this case, we define the following partition of $C_{\lambda,\mu}$:
        \begin{align*}
            \Omega_{1+}^{i} & =\{\omega\mid \omega(i)=i \}, \\
            \Omega_{1-}^{i} & = \{\omega\mid \omega(i)=-i \},\\
            \Omega_{2\phantom-}^i & = C_{\lambda,\mu}\setminus (\Omega_{1+}^i\cup \Omega_{1-}^i).
        \end{align*}
    Using Lemma~\ref{1cycles},
        \begin{align*}
            \pr_{\lambda,\mu}[\Omega_{1+}^i \cap \{\inv_{-i,i}=1\}] & =0=\frac{1}{2}\pr_{\lambda,\mu}[\Omega_{1+}^i]-\frac{{m_1(\lambda)}}{2n}, \\
            \pr_{\lambda,\mu}[\Omega_{1+}^i \cap \{\inv_{-i,i}=1\}] & =\pr_{\lambda,\mu}[\Omega_{1+}^i]=\frac{1}{2}\pr_{\lambda,\mu}[\Omega_{1-}^i]+\frac{{m_1(\mu)}}{2n}.
        \end{align*}
    Conjugation by $(-i)$ on $\Omega_2^i$ interchanges the sets \[\{\omega\in \Omega_2^i\mid \omega(i)=k,\omega(-i)=-k\} \quad \text{ and } \quad \{\omega\in \Omega_2^i\mid \omega(i)=-k,\omega(-i)=k\},\] and hence, \[\pr_{\lambda,\mu}[\Omega_2^i\cap \{\inv_{-i,i}=1\}]=\frac{1}{2}\pr_{\lambda,\mu}[\Omega_2^i].\]
    The result now follows by adding up these three terms.
\end{proof}

\begin{theorem}\label{thm:mean-des}
The mean of $\des_B$ on a conjugacy class $C_{\lambda,\mu}$ of $B_n$ is given by
    \[\mathbb{E}_{\lambda,\mu}[\des_B]=\frac{n}{2} - \frac{\Delta^2(\lambda,\mu)}{2n}.\]
\end{theorem}

\begin{proof}
    We begin by decomposing
    \[\E_{\lambda,\mu}[\des_B] = \E_{\lambda,\mu}[I_{-1,1}] + \E_{\lambda,\mu}\left[\sum_{i=1}^{n-1} \inv_{i,i+1}\right].\]
    The preceding proposition gives the first term. We use Theorem~\ref{thm:generalinversion} for the second term. For the weight function for $\des_B$, we have $\alpha_n(\des_B)=n-1$ and 
     $\beta_n(\des_B)=0$. Direct simplification shows 
     \begin{equation*}
         \begin{split}
             \E_{\lambda,\mu}[\des_B] & =  \frac{1}{2}-\frac{\Delta^1(\lambda,\mu)}{2n}+ \frac{n-1}{2}\left( 1+\frac{\Delta^1(\lambda,\mu)-\Delta^2(\lambda,\mu)}{n(n-1)}\right) \\
             & = \frac{n}{2} - \frac{\Delta^2(\lambda,\mu)}{2n}. \qedhere
         \end{split}
     \end{equation*}
\end{proof}

We now consider $\negative$ on conjugacy classes of $B_n$. Note that a special case for the mean of $\negative$ was previously given in Example \ref{rem:negsum}.

\begin{lemma}\label{lem:neg}
    For fixed $i\in [n]$, $j\in [\pm n]\setminus \{\pm i\}$, we have that \[\pr_{\lambda,\mu}[\omega(i)=j\mid \omega(i)\neq  \pm i]=\frac{1}{2(n-1)}.\]
    Consequently, we have that
    \[\pr_{\lambda,\mu}[\omega(i)=j]= \left(1-\frac{{m_1(\lambda)}}{n}-\frac{{m_1(\mu)}}{n}\right)\cdot \frac{1}{2(n-1)}.\]
\end{lemma}

\begin{proof}
    Consider conjugation by the permutation $\tau=(1,2,\ldots,\widehat{i},\ldots,n)$.
When $\omega(i)\in [n]\setminus \{i\}$, 
 \[\tau\omega\tau^{-1}(i)=
 \begin{cases} 
 \omega(i)+1 & \text{ if $\omega(i)\notin \{i-1,n\}$} \\
 i+1 & \text{ if $\omega(i)=i-1$} \\
 1 & \text{ if $\omega(i)=n$}.
 \end{cases}\]
Hence, this conjugation induces bijections among the sets $\{\omega\mid \omega(i)=k\}$ for $k\in [n]\setminus \{i\}$. Similarly, analyzing cases when $\omega(i)\notin [n]\cup \{-i\}$ implies this induces bijections among $\{\omega\mid \omega(i)=-k\}$ for $k\in [n]\setminus \{i\}$. Finally, if we conjugate by the odd $1$-cycle $\tau_k=(-k)$ for $k\in [n]\setminus \{i\}$, we see that when $\omega(i)=k$, \[\tau_k\omega \tau_k(i)=-k.\]
Hence, this induces a bijection between $\{\omega\in C_{\lambda,\mu}\mid \omega(i)=k\}$ and $\{\omega\in C_{\lambda,\mu}\mid \omega(i)=-k\}$.
Then we can decompose $\{\omega\in C_{\lambda,\mu}\mid \omega(i)\neq \pm i\}$ into
 $2(n-1)$ sets of the same size based on the image of $i$, implying 
 \[\pr_{\lambda,\mu}[\omega(i)=j\mid \omega(i)\neq \pm i]=\pr_{\lambda,\mu}[\omega(i)=-j\mid \omega(i)\neq \pm i]=\frac{1}{2(n-1)}.\]
The second claim now follows by using this with Lemma \ref{1cycles} on \[\pr_{\lambda,\mu}[\omega(i)=j]=\pr_{\lambda,\mu}[\omega(i)\neq \pm i]\cdot \pr_{\lambda,\mu}[\omega(i)=j\mid \omega(i)\neq \pm i].\qedhere \]
\end{proof}

\begin{proposition}\label{prop:neg}
The mean of $\negative$ on a conjugacy class $C_{\lambda,\mu}$ of $B_n$ is given by
\[\E_{\lambda,\mu}[\negative]=-\frac{n(n+1)}{4} + \frac{n+1}{4} \Delta^1(\lambda,\mu).\]
\end{proposition}

\begin{proof}    
First note  the identities 
\[\sum_{i=1}^n i=\binom{n+1}{2} \quad \text{ and } \quad \sum_{i=1}^n \sum_{j=1, j\ne i}^n j = \sum_{i=1}^n \left(\binom{n+1}{2}-i\right)=(n-1)\binom{n+1}{2}.\]
    Applying Lemmas \ref{1cycles} and \ref{lem:neg}, we find
    \begin{equation*}
    \begin{split}
       \E_{\lambda,\mu}[\negative]
        & = \E_{\lambda,\mu}\left[\sum_{i=1}^n \sum_{j=1}^n (-j)\cdot I_{(i,-j)}(\omega)\right]\\
        & = -\sum_{i=1}^n \left( i\cdot \pr_{\lambda,\mu}[\omega(i)=-i] + \sum_{j\neq i} j\cdot \pr_{\lambda,\mu}[\omega(i)=-j]\right) \\
        & =  -\frac{{m_1(\mu)}}{n}\left(\sum_{i=1}^n i \right)- \left(1-\frac{{m_1(\lambda)}}{n}-\frac{{m_1(\mu)}}{n}\right)\cdot \frac{1}{2(n-1)}\sum_{i=1}^n\sum_{j\neq i} j \\
        &=-m_1(\mu) \frac{n+1}{2}
        -(n-m_1(\lambda)-m_1(\mu))\frac{n+1}{4}\\
        &=-\frac{n(n+1)}{4} + \frac{n+1}{4} \Delta^1(\lambda,\mu).\qedhere
        \end{split}
\end{equation*}
\end{proof}

Using the above result, we now give explicit formulas for the mean of $\inv$ and $\inv_B$ on conjugacy classes of $B_n$.

\begin{theorem}\label{thm:inversions}
The mean of $\inv$ and $\inv_B$ on a conjugacy class $C_{\lambda,\mu}$ of $B_n$ are given by
\begin{equation*}
    \begin{split}
\E_{\lambda,\mu} [\inv]     
&=\frac{n(n-1)}{4} -\frac{n-3}{12}\Delta^1(\lambda,\mu)  -\frac{1}{6} \Delta^2(\lambda,\mu), \\
        \E_{\lambda,\mu} [\inv_B] 
        & 
        = \frac{n^2}{2}-\frac{n}{3}\Delta^1(\lambda,\mu) -\frac{1}{6}\Delta^2(\lambda,\mu).
    \end{split}
\end{equation*}

\end{theorem}

\begin{proof}
    For $\inv$, we apply Theorem \ref{thm:generalinversion}. Observe that this statistic corresponds to $\wt_n(i,j)=1$ for all $1\le i<j\le n$, and 0 otherwise. Then $\alpha(\inv)=\binom{n}{2}$, and  
    a standard binomial identity implies 
    \[\beta(\inv)=\sum_{1\leq i< j\leq n} (j-i-1)=\binom{n}{3}.\]
    Consequently, we obtain 
 \[\E_{\lambda,\mu}[\inv] =\frac{n(n-1)}{4}  \left(1 +
\frac{\Delta^1(\lambda,\mu)-\Delta^2(\lambda, \mu)}{n(n-1)} \right)
-\frac{n(n-1)}{12}\left( 
\frac{\Delta^1(\lambda,\mu)}{n-1}- 
\frac{\Delta^2(\lambda,\mu)}{n(n-1)}
\right),\]
which simplifies to the stated formula. Recalling (\ref{eqn:InvA-InvB-Neg}),
the result for $\E_{\lambda,\mu}[\inv_B]$ follows  by subtracting the expression for $\E_{\lambda,\mu}[\negative]$ in Proposition \ref{prop:neg} from the expression for $\E_{\lambda,\mu}[\inv]$ above. 
\end{proof}

\begin{remark}\label{rem:other-stats} Explicit expressions for the first moment of other statistics on $C_{\lambda,\mu}$, e.g., the major index and the flag major index defined in \cite{AdinBrentiRoichman}, can be obtained in this manner.
\end{remark}

\begin{corollary}\label{cor:inversions}
Let $X$ be any of the following statistics: $\des_B$, $\negative$, $\inv$, or $\inv_B$. Then $\E_{\lambda,\mu}[X]$ coincides on all conjugacy classes where $m_1(\lambda)=m_1(\mu)$. In particular, this holds when all cycles have length at least $2$ or when $\lambda=\mu$.
\end{corollary}

The statistics $\inv$, $\inv_B$, and $\des_B$ have degree aat most $2$, so by Theorem \ref{thm:IndependenceHigherMoments}, we expect their means to coincide on all $C_{\lambda,\mu}$ where $\lambda$ and $\mu$ have no parts of size $1$ or $2$. However, the above corollary shows a stronger statement. One natural question is whether $\inv_B$ and $\des_B$ actually have degree $1$. We show that in general, this is not the case. 

\begin{theorem}\label{thm:size2}
    For $n\geq 3$, the statistics $\inv$, $\inv_B$, and $\des_B$ on $B_n$ have degree $2$.
\end{theorem}

\begin{proof}
    Example~\ref{ex:des-inv-negsum-constraints} shows these statistics have degree at most $2$. To rule out the case of degree $1$, we show that these statistics cannot be expressed in the form
    \begin{equation}\label{eq:size1}
        c+\sum_{i\in [n],j\in [\pm n]} c_{(i,j)}I_{(i,j)}.
    \end{equation}
    where $c,c_{(i,j)}\in \mathbb{R}$. 
    Note that here $I_{(i,j)}$ is the indicator function for the partial colored permutation $(i,j)$ as defined in Section 3. 
    
    We first consider $\inv$ and $\inv_B$. Observe that for permutations $\omega$ with $\omega(i)>0$ for all $i$, any such decomposition must satisfy
    \[\inv(\omega)=\inv_B(\omega)=c+\sum_{i\in [n]} c_{(i,\omega(i))},\]
    where equality of the inversion statistics follows from the fact that $\negative(\omega)=0$. We define the two sets of permutations in $B_n$ using the one-line notation
    \begin{equation*}
    \begin{split}
        P_1=\{[1,2,3,4,5,\ldots,n],[2,3,1,4,5,\ldots,n],[3,1,2,4,5,\ldots,n]\}, \\ 
    P_2=\{[3,2,1,4,5,\ldots,n],[2,1,3,4,5,\ldots,n],[1,3,2,4,5,\ldots,n]\}.
    \end{split}
    \end{equation*}
    By construction, any expression of the form (\ref{eq:size1}) would imply
    \[\sum_{\omega\in P_1} \inv(\omega) = c+\sum_{i\in [3],j\in [3]} c_{(i,j)}+3\cdot \sum_{i\in [n]\setminus [3]} c_{(i,i)} = \sum_{\omega\in P_2} \inv(\omega).\]
    However, a direct calculation shows that 
    \[\sum_{\omega\in P_1} \inv(\omega)=4 \neq 5=\sum_{\omega\in P_2} \inv(\omega).\]
    Hence, $\inv$ and $\inv_B$ cannot be expressed in the form shown in (\ref{eq:size1}).

    The statistic $\des_B$ follows from a similar argument. If a decomposition using size $1$ partial colored permutations exists, then $\sum_{\omega\in P_1} \des_B(\omega)  = \sum_{\omega\in P_2} \des_B(\omega).$
    We similarly calculate 
    \[ \sum_{\omega\in P_1} \des_B(\omega)= 2 \neq 4 =\sum_{\omega\in P_2} \des_B(\omega). \qedhere \]
\end{proof}

\begin{corollary}\label{cor:conv-IndependenceHigherMoments}
    The converse of Theorem \ref{thm:IndependenceHigherMoments} does not hold. 
\end{corollary}

\begin{proof}
Theorem \ref{thm:size2} implies that the inversion and descent statistics on $B_n$  have degree $2$ when $n\geq 3$. Corollary \ref{cor:inversions} implies that the means of these statistics coincide on all $C_{\lambda,\mu}$ with no cycles of length $1$.
\end{proof}

\begin{remark} 
    We note that the statistics $\des_B$, $\inv$, and $\inv_B$ on $B_2$ have degree $1$. Letting $X$ denote any of these statistics, one can set up a system of linear equations for the conditions that the coefficients $c_{(i,j)}$ must satisfy for \[X=c+\sum_{i\in [2],j\in [\pm 2]} c_{(i,j)}I_{(i,j)}\] to hold on every element $\omega\in B_2$. For these statistics, this resulting system of equations has a solution (and in fact, infinitely many solutions). Hence, the assumption $n\geq 3$ in Theorem~\ref{thm:size2} is necessary.
\end{remark}

We conclude with the following analogue of \cite[Proposition 6.2]{GRWCPermutationStatistics}.  
\begin{proposition}\label{prop:first-moments-allBn} We have the following expressions for the first moments over all of $B_n$:
\begin{enumerate}
    \item[(1)] $\E_{B_n}[\inv]=\frac{n(n-1)}{4}$,
    \item[(2)] $\E_{B_n}[\inv_B]=\frac{n^2}{2}$,
    \item[(3)] $\E_{B_n}[\des_B]=\frac{n}{2}$, and
    \item[(4)] $\E_{B_n}[\negative]=-\frac{n(n+1)}{4}$.
    \end{enumerate}
Let $X$ denote any of the above statistics. Then there is a  polynomial $f^{X}_n(a_1, b_1)$   of degree at most $2$ in $a_1$ and $b_1$ (degree 1 for $\negative$) such that 
\begin{equation} \sum_{(\lambda,\mu)\vdash n}  z^{-1}_{(\lambda, \mu)}  f^{X}_n(m_1(\lambda), m_1(\mu))=0\quad 
\mathrm{ and }\quad 
\E_{\lambda,\mu}[X]=\E_{B_n}[X]+
f^{X}_n(m_1(\lambda), m_1(\mu)).\end{equation}
\end{proposition}

\begin{proof} These facts follow 
from the formulas of Theorem~\ref{thm:mean-des},  and Theorem~\ref{thm:inversions}. Proposition~\ref{prop:neg} gives the result for $\E_{B_n}[\negative]$ 
using Lemma~\ref{lem:symm-Bnconjugacyclasses}.
\end{proof}

\section{Descents in conjugacy classes of {$B_n$} and a central limit theorem}\label{section:CLT}

In this section, we will prove Theorem~\ref{thm:CLTMain}. In the process, we will derive properties involving the generating function for $\des_B+1$ on a conjugacy class $C_{\lambda,\mu}$, as well as a formula for the number of elements in $C_{\lambda,\mu}$ with a fixed number of descents.

We first recall a description of the descent statistic
\[\des_B(\omega)=|\{i\in \{0\}\cup [n-1]\mid \omega(i)>\omega(i+1)\}|\]
with the convention that $\omega(0)=0$. Note that this is based on the usual order on $[\pm n]$. Reiner \cite{ReinerEuropJC1993-2} uses a different notion of descents: under the ordering 
\begin{equation}\label{eq:reiner_descents}
    1<2<\cdots<n<-n<\cdots<-2<-1,
\end{equation} $\omega$ has a descent at position $i\in [n-1]$ if $\omega(i)>\omega(i+1)$, and $\omega$ has a descent at position $n$ if $\omega(n)<0$. While the two definitions are different, \cite[Remark 5.1]{FulmanKimLeePetersen2021} shows that the generating function
\begin{equation}\label{eq:des_gen_function}
    \sum_{\omega \in B_n} t^{\des_B(\omega)+1} \prod_i x_i^{m_i(\lambda(\omega))} y_i^{m_i(\mu(\omega))}
\end{equation}
    is unaffected, where $(\lambda(\omega),\mu(\omega))$ denotes the cycle type of $\omega$.

Following Fulman's analysis for the symmetric group in \cite{FulmanJCTA1998}, our approach for establishing Theorem~\ref{thm:CLTMain} involves analyzing the generating function given in (\ref{eq:des_gen_function}), which allows us to analyze the generating function for $\des_B$ on a conjugacy class. We then relate this with the generating function for descents on all of $B_n$. In the case where there are no short cycles in $C_{\lambda,\mu}$, we will ultimately conclude that certain moments of $\des_B$ agree on $C_{\lambda,\mu}$ and $B_n$, and this allows us to use the method of moments with a known central limit theorem of Chow and Mansour for $\des_B$ on $B_n$ given below.

\begin{proposition}\cite[Thm 3.4]{CM2012}\label{prop:norm-allBn} Let $X_n$ be the descent statistic $\des_B$ defined on $B_n$. Then $X_n$ has mean $n/2$ and variance $(n+1)/12$, and as $n\to\infty$, the standardized random variable $(X_n-n/2)/\sqrt{(n+1)/12}$ converges to a standard normal distribution.
    \end{proposition}

   We will need the generating function of $\des_B$ over all of $B_n$, which is well-understood.

     \begin{proposition}\cite[Eqn.~(13.3)]{Petersen2015}\label{prop:Bn_genfunction} Let ${B}_n(t)=\sum_{\omega \in B_n} t^{\des_B(\omega)+1}$. 
    Then 
    \begin{equation*}
        \frac{{B}_n(t)}{(1-t)^{n+1}}=\sum_{k \geq 1} (2k-1)^n t^k.
    \end{equation*}
    \end{proposition}

\begin{remark}
Chow and Mansour established Proposition~\ref{prop:norm-allBn} for the general descent statistic $\des_{n,r}$ on $\mathfrak{S}_{n,r}$, where $r$ is fixed and $n\to\infty$. The definition of $\des_{B}$ on $B_n$ is not equivalent to the definition of $\des_{n,2}$ on $\mathfrak{S}_{n,2}$ under the usual isomorphism between these groups. However, Steingr\'{i}msson \cite{steingrim94} shows that the generating function for $\des_{n,2}+1$ matches the one given in Proposition~\ref{prop:Bn_genfunction}, so the distributions of $\des_{B}$ and $\des_{n,2}$ are the same on $B_n\cong \mathfrak{S}_{n,2}$. This justifies our restatement of Proposition~\ref{prop:norm-allBn} with $\des_B$.
\end{remark}

We now analyze (\ref{eq:des_gen_function}), which will allow us to derive an expression for the generating function of $\des_B$ on a conjugacy class $C_{\lambda,\mu}$. We will need several known results, which we  state now. We begin with a definition for an expression that will come up frequently in our analysis.

\begin{definition}\label{def:exponent-Reinernecklaces}
Let  $\mu(d)$ be the number-theoretic M\"{o}bius function.  For nonnegative integers $r$ and $m$, define \begin{equation}\label{eq:mobius}
N(r, 2m)=\frac{1}{2m} \sum_{\substack{d\vert m \\ d \text{ odd}}} \mu(d) \left(r^{m/d}-1\right).
\end{equation} 
\end{definition}

\begin{remark} \label{rmk:Reinernecklaces} 
 Reiner \cite[Theorem 4.1 and Theorem 4.2]{ReinerEuropJC1993-2} 
describes two sets of objects, \emph{primitive blinking necklaces} and \emph{primitive twisted necklaces}. Reiner then shows that if
$D_{m, i}^{(k)}$ is the number of primitive blinking necklaces $D$ of size $m$ with $|D|=i$ and $\max(D) \leq k-1$, and $P_{m, i}^{(k)}$ is the number of primitive twisted necklaces $P$ of size $m$ with $|P|=i$ and $\max(P) \leq k-1$, then 
\begin{equation}\label{eqn:Dmi}
\sum_{i=0}^{(k-1)m} D_{m, i}^{(k)}=\begin{cases}
N(2k-1, 2m) & \text{ if } k>0 \text{ and } m>1\\
k & \text{ if } k>0 \text{ and } m=1\\
0 & \text{ if } k=0,
\end{cases}
\end{equation}
\begin{equation}\label{eqn:Pmi}
\sum_{i=0}^{(k-1)m} P_{m, i}^{(k)}=\begin{cases}
N(2k-1,2m) & \text{ if } k>0 \text{ and } m>1 \\
k-1 & \text{ if } k>0 \text{ and } m=1\\
0 & \text{ if } k=0.
\end{cases}
\end{equation}
Consequently, $N(2k-1, 2m)$ enumerates a set of objects, so it must be a nonnegative integer for all $k,m \ge 1$.
\end{remark}

We use the notation $(\lambda(\omega),\mu(\omega))$ for the cycle type of $\omega\in B_n$. The following result, which  is stated without proof in \cite[Theorem 5.3]{FulmanKimLeePetersen2021}, is a special case of \cite[Theorem 4.1]{ReinerEuropJC1993-2}.  In view of some possibly confusing typos in these papers, and the fact that this result plays a key role in our analysis, we include a proof.

\begin{theorem}[cf. {\cite[Theorem~5.3]{FulmanKimLeePetersen2021}}]\label{thm:all}
 The following holds:
\begin{equation}
\begin{split}
    & \,\sum_{n \geq 0} \frac{u^n}{(1-t)^{n+1}} \left( \sum_{\omega \in B_n} t^{\des_B(\omega)+1} \prod_i x_i^{m_i(\lambda(\omega))} y_i^{m_i(\mu(\omega))}\right) \\
    =& \, 1+\sum_{k\geq 1} t^k \frac{1}{1-x_1 u} \prod_{m \geq 1} \left(\frac{1+y_m u^m}{1-x_m u^m}\right)^{N(2k-1, 2m)}.
\end{split}
\end{equation}
\end{theorem}

\begin{proof}
In \cite[Theorem 4.1]{ReinerEuropJC1993-2}, 
it was established that
\begin{align}\label{eq:der1}
&\sum_{n \geq 0} \frac{u^n}{(t; q)_{n+1}} \left( \sum_{\omega \in B_n} t^{\des(\omega)+1} q^{\maj(\omega)} \prod_i x_i^{m_i(\lambda(\omega))} y_i^{m_i(\mu(\omega))}\right)\notag \\
=&\sum_{k\geq 0} t^k \prod_{m \geq 1} \prod_{i=0}^{(k-1)m} (1-x_m u^m q^i)^{-D_{m, i}^{(k)}} (1+y_m u^m q^i)^{P_{m, i}^{(k)}},
\end{align}
where
\begin{itemize}
\item 
$(t; q)_{n+1}=(1-t)(1-tq)\cdots (1-tq^{n}),$
\item
$\des(\omega)$ and $\maj(\omega)$ are the descent and major index statistics with respect to the ordering in \eqref{eq:reiner_descents}, and
\item
$D_{m, i}^{(k)}$, $P_{m, i}^{(k)}$ are the nonnegative integers defined by Reiner as described in Remark~\ref{rmk:Reinernecklaces}. 
\end{itemize}
Setting $q=1$, (\ref{eq:der1}) simplifies to
\begin{align}\label{eq:der2}
&\sum_{n \geq 0} \frac{u^n}{(1-t)^{n+1}} \left( \sum_{\omega \in B_n} t^{\des(\omega)+1} \prod_i x_i^{m_i(\lambda(\omega))} y_i^{m_i(\mu(\omega))}\right)\notag \\
=&\sum_{k\geq 0} t^k \prod_{m \geq 1}\biggr[ \biggr( (1-x_m u^m)^{-\sum_{i=0}^{(k-1)m} D_{m, i}^{(k)}} \biggr) \cdot \biggr((1+y_m u^m)^{\sum_{i=0}^{(k-1)m} P_{m, i}^{(k)}}\biggr) \biggr] \\
=& 1+\sum_{k\geq 1} t^k \prod_{m \geq 1}\biggr[ \biggr( (1-x_m u^m)^{-\sum_{i=0}^{(k-1)m} D_{m, i}^{(k)}} \biggr) \cdot \biggr((1+y_m u^m)^{\sum_{i=0}^{(k-1)m} P_{m, i}^{(k)}}\biggr) \biggr].\notag
\end{align}
As noted in the discussion surrounding \eqref{eq:reiner_descents} and \eqref{eq:des_gen_function}, we can replace $\des$ in the first expression of \eqref{eq:der2} with $\des_{B}$. The result then follows by applying Remark~\ref{rmk:Reinernecklaces} on the last expression, noting that in the boundary case $m=k=1$, the exponent of $(1-x_1u)$ is  $N(2k-1, 2)+1$ from~\eqref{eqn:Dmi}, while the exponent of $(1-y_1u)$ is precisely $N(2k-1, 2)$ from~\eqref{eqn:Pmi}. 
\end{proof}

For a fixed bi-partition $(\lambda,\mu)$ of $n$, we can now derive the following expression for the generating function $B_{\lambda,\mu}(t)$  $=\sum_{\omega \in C_{\lambda,\mu}} t^{\des_B(\omega)+1}$ of descents over the conjugacy class $C_{\lambda,\mu}$.

\begin{proposition}\label{prop:all}
Let 
$\lambda=(1^{{m_1(\lambda)}},2^{{m_2(\lambda)}},\ldots)$ and $\mu=(1^{{m_1(\mu)}},2^{{m_2(\mu)}},\ldots)$. Then 
$B_{\lambda,\mu}/(1-t)^{n+1}$ can be expressed as
\begin{equation} \label{eq:conjugacy1}
\begin{split}
     t \delta_{((1^n), \emptyset)}+\sum_{k \geq 2} & t^k  \left(\prod_{i\ge 1} \binom{N(2k-1, 2i)}{m_i(\mu)} \,
\prod_{i\ge 2} \binom{N(2k-1, 2i)+m_i(\lambda)-1}{m_i(\lambda)}\right)   \binom{N(2k-1, 2)+m_1(\lambda)}{m_1(\lambda)},\\
\end{split}
\end{equation}
which is also equivalent to
\begin{equation*}
\, t \delta_{((1^n), \emptyset)}+\sum_{k \geq 2} t^k \frac{m_1(\lambda)+k-1}{k-1} \prod_{i \geq 1} \binom{N(2k-1, 2i)-1+m_i(\lambda)}{m_i(\lambda)} \binom{N(2k-1, 2i)}{m_i(\mu)}.
\end{equation*}
Here $\delta_{((1^n), \emptyset)}$ is the Kronecker delta which is 1 for the conjugacy class  $\lambda,\mu=((1^n), \emptyset)$, and zero otherwise.
\end{proposition}

\begin{proof}
We set $u=1$ in the generating function of Theorem \ref{thm:all} and extract the coefficients of $x_i$ and $y_i$ for all $i$. We then examine the coefficient of $t^k$ in the right-hand side of Theorem \ref{thm:all}.  Noting that $N(1, 2i)=0$, this coefficient can then be written as  
\begin{equation*} 
\begin{split}
    & \, \frac{1}{1-x_1} \prod_{i \geq 1} \left(\frac{1+y_i}{1-x_i}\right)^{N(2k-1, 2i)}
= \left(\prod_{i \geq 1}(1+y_i)^{N(2k-1, 2i)} \, \prod_{i \geq 2}(1-x_i)^{-N(2k-1, 2i)}\right) 
(1-x_1)^{-N(2k-1, 2)-1}.
\end{split}
\end{equation*}
For $k=1$ the contribution comes only from powers of $x_1$. For $k\ge 2,$ the coefficient of the product $\prod_i x_i^{m_i(\lambda)} y_i^{m_i(\mu)}$ is thus 
\begin{equation*}
\begin{split}
    & \left(\prod_{i\ge 1} \binom{N(2k-1, 2i)}{m_i(\mu)} \,
\prod_{i\ge 2} \binom{N(2k-1, 2i)+m_i(\lambda)-1}{m_i(\lambda)}\right)\cdot\binom{N(2k-1, 2)+1+m_1(\lambda)-1}{m_1(\lambda)},
\end{split}
\end{equation*}
where we have used the  binomial formula 
$(1-s)^{-m}=\sum_{j\ge 0}s^j \binom{m+j-1}{j}$ for $m>0$. This gives the first expression \eqref{eq:conjugacy1}. 
Using $\binom{a+1}{r+1}=\frac{a+1}{r+1}\binom{a}{r}$  and the fact that $N(2k-1,2)=k-1$ in the third factor above, the second expression follows.
\end{proof}

\begin{corollary}
The number of permutations $\omega \in B_n$ that are of cycle type $(\lambda,\mu)$ and have $d-1$ descents is
\begin{equation}
\begin{split}
    & \, \sum_{k=1}^d (-1)^{d-k} \binom{n+1}{d-k} \binom{m_1(\lambda)+k-1}{m_1(\lambda)} \cdot \prod_{i\ge 2} \binom{N(2k-1, 2i)+m_i(\lambda)-1}{m_i(\lambda)} \prod_{i\ge 1} \binom{N(2k-1, 2i)}{m_i(\mu)}.
\end{split}
\end{equation}
\end{corollary}

\begin{proof}
We proceed as in the proof of Proposition \ref{prop:all}, but multiply the expressions by $(1-t)^{n+1}$ to obtain the generating function $B_{\lambda,\mu}(t)$. The result follows by taking into account the coefficient of $t^d$.
\end{proof}

\begin{remark}
For $n=0$, we adopt the convention that $B_0$ is the trivial group containing a single permutation $\omega$ with $\des_B(\omega)=0$ and $m_i(\lambda)=m_i(\mu)=0$ for all $i$, for the single trivial conjugacy class $(\emptyset, \emptyset)$. We see that Proposition~\ref{prop:all} still holds:
$$B_{\emptyset, \emptyset}(t)=t=(1-t) \sum_{k \geq 1} t^k.$$
\end{remark}

In certain cases, the expression from Proposition \ref{prop:all} simplifies. We give some examples in the next result.

\begin{proposition}\label{prop:SpecialCases}
We have the following for $n\ge 1$:
\begin{enumerate}
    \item $B_{((1^n),\emptyset)}(t)=t$,
    \item $B_{(\emptyset, (1^n))}(t)=t^{n+1}$,
    \item for $n\ge 2$, $B_{((n),\emptyset)}(t)=(1-t)^{n+1} \sum_{k\ge 2} t^k N(2k-1,2n)$,
    \item  $B_{(\emptyset, (n))}(t)=(1-t)^{n+1} \sum_{k\ge 2} t^k N(2k-1,2n)$, and
    \item if $\lambda$ and $\mu$ both have distinct parts and $m_1(\lambda)=1=m_1(\mu)$, then 
    \[B_{\lambda,\mu}(t)=B_{\mu, \lambda}(t).\] 
\end{enumerate}
\end{proposition}

\begin{proof} From \eqref{eq:mobius}, we see that \[N(2k-1,2)=k-1, \, N(2k-2,4)=k(k-1) \text{ for } k\ge 2, \text{ and } N(1, 2m)=0 \text{ for } m\ge 1.\] 
We will also need the binomial formula for negative exponents
\[(1-s)^{-m}=\sum_{j\ge 0}s^j \binom{m+j-1}{j}.\]
\begin{enumerate}
    \item Here $m_1(\lambda)=n$, $m_j(\lambda)=0$ for $j\ne 1$, and $m_i(\mu)=0$ for $i\ge 1$. Thus  
    \begin{equation*}
        \begin{split}
            \frac{B_{((1^n),\emptyset)}(t)}{(1-t)^{n+1}}
            &= t+\sum_{k\ge 2} t^k \frac{n+k-1}{k-1} \binom{n+k-2}{n}\\
            &=t\left(1+\sum_{k\ge 2} t^{k-1}\binom{n+k-1}{k-1}\right)\\
            &=t\left(1+\sum_{j\ge 1}t^{j} \binom{(n+1)+j-1}{j}\right) \\
            & =t(1-t)^{-(n+1)}.
        \end{split}
    \end{equation*}
    \item Now we have $m_i(\lambda)=0$ for $i\ge 1$,
    and $m_1(\mu)=n,$ and $m_j(\mu)=0$ for $j\ne 1$.  Since $N(2k-1,2)=k-1$, we obtain 
    \begin{equation*}
         \begin{split}
            \frac{B_{(\emptyset, (1^n))}(t)}{(1-t)^{n+1}}
            &=\sum_{k\ge 2}t^k\binom{k-1}{n}\\
            &=t^{n+1}\sum_{k\ge n+1}t^{k-n-1}\binom{(n+1)+(k-n-1)-1}{k-n-1} \\
            & =t^{n+1} (1-t)^{-(n+1)}.\\
        \end{split}
    \end{equation*}
    \item Let $n\ge 2$. Now $m_n(\lambda)=1$, $m_i(\lambda)=0$ for $i\ne 1$, while $m_i(\mu)=0$ for $i\ge 1.$  The result follows.
    \item This follows similarly.
    \item In this case we have $0\le m_i(\lambda), m_j(\mu)\le 1$ for $i,j\ge 1$, and $m_1(\lambda)=m_1(\mu)=1$. Then 
    \begin{equation*}
         \begin{split}
            \frac{B_{\lambda,\mu}(t)}{(1-t)^{n+1}}
            & =\sum_{k\ge 2} t^k \frac{k}{k-1} 
            \prod_{i\ge 1, m_i(\lambda)=1}N(2k-1, 2i)
            \prod_{j\ge 1, m_j(\mu)=1} N(2k-1,2j) \\
            & 
            =\frac{B_{\mu, \lambda}(t)}{(1-t)^{n+1}}. \qedhere
        \end{split}
    \end{equation*}
\end{enumerate}
\end{proof}   

We next derive an elegant analogue of a result of Fulman \cite[Proof of Theorem 2]{FulmanJCTA1998}, which will relate $B_{\lambda,\mu}(t)$ and $B_n(t)$. Recall from Definition~\ref{def:Delta-biptn} that $\Delta^2(\lambda,\mu)=m_1(\lambda)^2-m_1(\mu)^2$, and from Lemma~\ref{lem:classsize}, we have $$|C_{\lambda,\mu}|=\frac{2^n n!}{2^{\ell(\lambda)} z_\lambda 2^{\ell(\mu)} z_\mu},$$ where $\ell(\lambda)=\sum_i  m_i(\lambda)$, $\ell(\mu)=\sum_i  m_i(\mu)$, $z_\lambda=\prod_i i^{m_i(\lambda)} m_i(\lambda)!$, and $z_\mu=\prod_i i^{m_i(\mu)} m_i(\mu)!$. Hence, from Definition~\ref{def:doublezee},
 \[z_{\lambda,\mu}=2^{\ell(\lambda)} z_\lambda 2^{\ell(\mu)} z_\mu=\prod_{i\geq 1} (2i)^{m_i(\lambda)+m_i(\mu)} m_i(\lambda)! m_i(\mu)!.\]
Finally, let $s_n^i$ be the Stirling number of the first kind, whose absolute value is the number of permutations in $\mathfrak{S}_n$ with $i$ disjoint cycles. The following generating functions  are well known  (see \cite[Chapter 1]{StanEC2}):
\begin{equation}\label{lm:stirling}
\begin{split}
&\sum_{i=1}^n s_n^i y^i=y(y-1)\cdots (y-n+1)=n!\,\binom{y}{n},\\
&\sum_{i=1}^n |s_n^i| y^i=y(y+1)\cdots(y+n-1)=
n!\,\binom{y+n-1}{n}.
\end{split}
\end{equation}

\begin{theorem}\label{thm:expansion}
Let $C_{\lambda,\mu}$ be a conjugacy class of $B_n$, let $B_n(t)=\sum_{\omega \in B_n} t^{\des_B(\omega)+1}$, and let $B_{\lambda,\mu}(t)=\sum_{\omega \in C_{\lambda,\mu}} t^{\des_B(\omega)+1}$. Then
\begin{equation}\label{eq:B}
\frac{B_{\lambda,\mu}(t)}{|C_{\lambda,\mu}|}=\frac{B_n(t)}{2^n n!}+\frac{1-t}{2n} \frac{B_{n-1}(t)}{2^{n-1} (n-1)!} \Delta^2(\lambda,\mu)+(1-t)^2 g(t),
\end{equation}
where $g(t)$ is some polynomial in $t$. 
\end{theorem}

\begin{proof}
From Proposition \ref{prop:Bn_genfunction}, we have
\begin{equation}
B_n(t)=(1-t)^{n+1} \sum_{k \geq 1}(2k-1)^n t^{k}.
\end{equation}
By Proposition \ref{prop:all} and Lemma~\ref{lem:classsize},
\begin{align}\label{eq:starting}
\frac{B_{\lambda,\mu}(t)}{|C_{\lambda,\mu}|} & =(1-t)^{n+1}\frac{z_{\lambda,\mu}}{2^n n!} \Biggl[t \delta_{((1^n), \emptyset)} \nonumber\\
& \phantom{=} +\sum_{k \geq 2} t^k \, \left(\prod_{i\ge 1} \binom{N(2k-1, 2i)}{m_i(\mu)} \,
\prod_{i\ge 2} \binom{N(2k-1, 2i)+m_i(\lambda)-1}{m_i(\lambda)}\right)\, \nonumber \\
& \phantom{=+\sum t^k \, } \cdot\binom{N(2k-1, 2)+m_1(\lambda)}{m_1(\lambda)}  \Biggr].
\end{align}
Using \eqref{lm:stirling} and \eqref{eq:starting}, the coefficient of $t^k$ for $k\ge 2$ is 
\begin{align}
&\left[\prod_{i\ge 1} \binom{N(2k-1, 2i)}{m_i(\mu)} \,
\prod_{i\ge 2} \binom{N(2k-1, 2i)+m_i(\lambda)-1}{m_i(\lambda)}\right]\,\cdot\binom{N(2k-1, 2)+m_1(\lambda)}{m_1(\lambda)} \notag\\
&=\left[\prod_{i\ge 1} \frac{\sum_{b=1}^{m_i(\mu)} s_{m_i(\mu)}^b (N(2k-1, 2i))^b}{m_i(\mu)!} \, \prod_{i\ge 2} \frac
{\sum_{a=1}^{m_i(\lambda)} |s_{m_i(\lambda)}^a| (N(2k-1, 2i))^a}{m_i(\lambda)!}\right] \notag \\
& \phantom{=} \cdot \left[ \frac{1}{m_1(\lambda)!} \sum_{a=1}^{m_1(\lambda)} |s_{m_1(\lambda)}^a| \left(\frac{(2k-1)+1}{2}\right)^a\right], \label{eq:a}
\end{align}
the last sum arising from the fact that $N(2k-1, 2)+1=k$. Combined with \eqref{eq:mobius}, we may view the coefficient of $t^k$ for $k\ge 2$ in \eqref{eq:starting} as a polynomial in $(2k-1)$. Therefore, we can view the entire expression for $B_{\lambda,\mu}$ as a polynomial in $(2k-1)$. 

The largest power of $(2k-1)$ in $B_{\lambda,\mu}$ is obtained by taking the largest power of $(2k-1)$ in each factor of \eqref{eq:a}. From  \eqref{eq:mobius}, $N(2k-1, 2i)$ is a polynomial in $(2k-1)$ with leading term $\frac{(2k-1)^i}{2i}$. It follows that  the largest power of $2k-1$ in  \eqref{eq:starting} occurs when we respectively take the summands corresponding to  $b=m_i(\mu)$ and $a=m_i(\lambda)$ in  the first two summations and $a=m_1(\lambda)$ in the third summation of \eqref{eq:a}, yielding
\begin{equation}\label{eq:degree_k}
\prod_{i\ge 1}\frac{(N(2k-1, 2i))^{m_i(\lambda)+m_i(\mu)}}{m_i(\lambda)!m_i(\mu)!} \sim \prod_{i\ge 1}\frac{\left(\frac{(2k-1)^i}{2i}\right)^{m_i(\lambda)+m_i(\mu)}}{m_i(\lambda)!m_i(\mu)!}.
\end{equation}
The term with the highest power of $(2k-1)$ in \eqref{eq:starting}, namely $(2k-1)^n$,  is therefore
\begin{align}
& \, \frac{z_{\lambda,\mu}}{2^n n!} (1-t)^{n+1} \sum_{{k \geq 2}} t^k \prod_{i \geq 1} \frac{\left(\frac{(2k-1)^i}{2i}\right)^{m_i(\lambda)+m_i(\mu)}}{m_i(\lambda)!m_i(\mu)!} \\
=& \,\frac{1}{2^n n!} (1-t)^{n+1} \sum_{{k \geq 2}} t^k (2k-1)^n=\frac{B_n(t)-t(1-t)^{n+1}}{2^n n!}.\nonumber
\end{align}

Again viewing~\eqref{eq:starting} as a polynomial in $(2k-1)$, now we identify the next leading term, i.e. the term with the second highest power of $(2k-1)$ in \eqref{eq:a}. We note that for $i>1$, setting $b\neq m_i(\mu)$ in the first summation of \eqref{eq:a}, or $a\neq m_i(\lambda)$ in the second summation of \eqref{eq:a}, or $d\neq 1$ in \eqref{eq:mobius} will all lower the degree on $(2k-1)$ in \eqref{eq:degree_k} by more than $1$. So for $i>1$, we still respectively set $b=m_i(\mu), a=m_i(\lambda)$ in each of the summations in \eqref{eq:a} and still  take $d=1$ in \eqref{eq:mobius}. 

The lowering of degree for $(2k-1)$ must therefore come from the $i=1$ term, by incorporating $(m_1(\lambda)+k-1)/(k-1)$ with the binomial coefficients. We take the next leading term of
\begin{equation}
\begin{split}
     & \, \binom{m_1(\lambda)+k-1}{m_1(\lambda)}\binom{k-1}{m_1(\mu)} 
     = \frac{\sum_{a=1}^{m_1(\lambda)} |s_{m_1(\lambda)}^a| \frac{1}{2^a}((2k-1)+1)^a}{m_1(\lambda)!} \frac{\sum_{b=1}^{m_1(\mu)} s_{m_1(\mu)}^b \frac{1}{2^b} ((2k-1)-1)^b}{m_1(\mu)!},
\end{split}
\end{equation}
which is given by
\begin{align}
&\left(|s_{m_1(\lambda)}^{m_1(\lambda)}|\cdot s_{m_1(\mu)}^{m_1(\mu)} \cdot \left(\frac{m_1(\lambda)-m_1(\mu)}{2}\right)+|s_{m_1(\lambda)}^{m_1(\lambda)-1}|\cdot s_{m_1(\mu)}^{m_1(\mu)}+|s_{m_1(\lambda)}^{m_1(\lambda)}|\cdot s_{m_1(\mu)}^{m_1(\mu)-1}\right) \notag \\ & \hspace{3cm} \cdot \frac{1}{m_1(\lambda)!m_1(\mu)!} \left(\frac{2k-1}{2}\right)^{m_1(\lambda)+m_1(\mu)-1}\notag \\
=&\left(\frac{m_1(\lambda)-m_1(\mu)}{2}+\binom{m_1(\lambda)}{2}-\binom{m_1(\mu)}{2}\right) \frac{1}{m_1(\lambda)!m_1(\mu)!} \left(\frac{2k-1}{2}\right)^{m_1(\lambda)+m_1(\mu)-1} \notag \\
=&\,\frac{1}{2m_1(\lambda)!m_1(\mu)!} \left(m_1(\lambda)^2-m_1(\mu)^2\right) \left(\frac{2k-1}{2}\right)^{m_1(\lambda)+m_1(\mu)-1}.
\end{align}
Hence the term with the second highest degree of $(2k-1)$  from the \eqref{eq:starting} is the term involving $(2k-1)^{n-1}$, namely, 
\begin{equation*}
\begin{split}
& \, \left(m_1(\lambda)^2-m_1(\mu)^2\right)\frac{1}{2^{n} n!} (1-t)^{n+1} \sum_{{k \geq 2}} t^k (2k-1)^{n-1} \\
=&\,\frac{1-t}{2n} \frac{B_{n-1}(t)-t(1-t)^n}{2^{n-1} (n-1)!} \Delta^2(\lambda,\mu)\\
=&\,\frac{1-t}{2n} \frac{B_{n-1}(t)}{2^{n-1} (n-1)!} \Delta^2(\lambda,\mu) -\frac{t(1-t)^{n+1}}{2^n n!}\Delta^2(\lambda,\mu) .
\end{split}
\end{equation*}

The result now follows by absorbing the term $\frac{t(1-t)^{n+1}}{2^n n!}\Delta^2(\lambda,\mu)$ into the terms with lower powers of $(2k-1)$.  For $r\ge 2$, each lower power $(2k-1)^{n-r}$ in \eqref{eq:starting} comes from the fraction 
$\frac{B_{n-r}(t)}{(1-t)^{n-r+1}}$
with a factor of $(1-t)^{n-r+1}$ in the denominator, 
and hence, because of the outer factor of $(1-t)^{n+1}$,  we obtain an  error term of the form $(1-t)^2 g(t)$ for some polynomial $g(t)$ as stated.
\end{proof}

\begin{remark}
Differentiating (\ref{eq:B}) with respect to $t$ and setting $t=1$, we obtain the first moment of descents in the conjugacy class $C_{\lambda,\mu}$. This agrees with Theorem~\ref{thm:mean-des}.
\end{remark}

In the special case where there are no short cycles, we obtain the following variation of Theorem~\ref{thm:expansion} that we use to establish our Central Limit Theorem on conjugacy classes. 

\begin{theorem}\label{thm:long_cycle}
Let $C_{\lambda,\mu}$ be a conjugacy class of $B_n$, let $B_n(t)=\sum_{\omega \in B_n} t^{\des_B(\omega)+1}$, and let $B_{\lambda,\mu}(t)=\sum_{\omega \in C_{\lambda,\mu}} t^{\des_B(\omega)+1}$. If $C_{\lambda,\mu}$ contains no cycles of length $1,2,\ldots,2\ell$, then
\begin{equation}
\frac{B_{\lambda,\mu}(t)}{|C_{\lambda,\mu}|}=\frac{B_n(t)}{2^n n!}+(1-t)^{\ell+1} g(t),
\end{equation}
where $g(t)$ is some polynomial in $t$. 
\end{theorem}

\begin{proof}
As in the proof of Theorem \ref{thm:expansion},
\begin{equation}
\begin{split}
&\frac{B_{\lambda,\mu}(t)}{|C_{\lambda,\mu}|}=\frac{1}{2^n n!}(1-t)^{n+1} \prod_{i>2\ell} (2i)^{m_i(\lambda)+m_i(\mu)} m_i(\lambda)! m_i(\mu)! \cdot \notag \\ & \hspace{2cm}\cdot \sum_{k \geq 1} t^k \prod_{i>2\ell} \binom{m_i(\lambda)+N(2k-1, 2i)-1}{m_i(\lambda)} \binom{N(2k-1, 2i)}{m_i(\mu)}.
\end{split}
\end{equation}
We note that, as a polynomial in $(2k-1)$, the leading term of $$\prod_{i>2\ell} \binom{m_i(\lambda)+N(2k-1, 2i)-1}{m_i(\lambda)} \binom{N(2k-1, 2i)}{m_i(\mu)}$$
is $(2k-1)^n$ as in the general case, but the second highest-degree term is at most $(2k-1)^{n-\ell-1}$ under the long cycle assumption. This is because a lower order term must have either some $a \neq m_i(\lambda)$ or some $b \neq m_i(\mu)$ in \eqref{eq:a}, or $d \neq 1$ in (\ref{eq:mobius}).  If some $a \neq m_i(\lambda)$ or some $b \neq m_i(\mu)$, then the power of $2k-1$ from such a term is at most $n-i<n-2\ell$. If some $d \neq 1$, 
 then the power of $2k-1$ from such a term is at most $n-(i-i/3)<n-\ell$. We then proceed similarly as in the proof of Theorem \ref{thm:expansion}.
\end{proof}

\begin{corollary}\label{cor:Bn}
Let $C_{\lambda,\mu}$ be a conjugacy class of $B_n$. If $C_{\lambda,\mu}$ contains no cycles of length $1,2,\ldots,2k$, then 
\[\E_{\lambda,\mu}[\des_B^k] = \E_{B_n}[\des_B^k].\]
\end{corollary}

We conclude with our central limit theorem for descents on conjugacy classes of $B_n$ without short cycles. To arrive at our conclusion, we appeal to the Method of Moments, stated below. This tool allows one to translate convergence of moments into convergence in distribution, provided that all moments are finite and the limiting distribution is uniquely determined by its moments. We will apply this when the limiting distribution is the standard normal distribution, which is uniquely determined by its moments. See \cite[Section 30]{billingsley} for further details. 

\begin{theorem}[Method of Moments]
    Let $\{X_n\}_{n\geq 1}$ be a sequence of random variables with finite $k$-th moment for all $k\geq 1$. Suppose that the distribution of $X$ is determined by its moments, and for each $k\geq 1$, the $k$-th moment of $X_n$ converges to the $k$-th moment of $X$ as $n\to\infty$. Then $X_n$ converges to $X$ in distribution. 
\end{theorem}

\begin{theorem}[Theorem~\ref{thm:CLTMain}]
For every $n\geq 1$, let $C_{\lambda_n, \mu_n}$ be a conjugacy class of $B_n$, and define $X_n$ to be the descent statistic $\des_B$ on $C_{\lambda_n,\mu_n}$. Suppose that for all $i$, the number of cycles of length $i$ in $\lambda_n$ and $\mu_n$ approaches 0 as $n\to\infty$. Then for sufficiently large $n$, $X_n$ has mean $n/2$ and variance $(n+1)/12$, and as $n\to\infty$, the random variable $(X_n-n/2)/\sqrt{(n+1)/12}$ converges in distribution to a standard normal distribution.
\end{theorem}

\begin{proof}
    The values for the mean and variance follow from applying Corollary~\ref{cor:Bn} and Proposition~\ref{prop:norm-allBn} on the first two moments of $X_n$ with the hypothesis that there are no cycles of length $1,2,3$, or $4$ when $n$ is sufficiently large. For the asymptotic behavior, we fix $k$ and expand
    \[\left(\frac{X_n-n/2}{\sqrt{(n+1)/12}}\right)^k = \frac{1}{((n+1)/12)^{k/2}} \sum_{i=0}^k {k\choose i} \left( \frac{n}{2}\right)^{k-i}  X_n^i.\]
    Linearity of expectation then implies that the $k$th moment of $(X_n-n/2)/\sqrt{(n+1)/12}$ on 
    $B_n$ or $C_{\lambda_n,\mu_n}$ can be expressed in terms of the first $k$ moments of $X_n$ on those respective sets.  Applying Corollary~\ref{cor:Bn} with the hypothesis that there are no cycles of length $1,2,\ldots,2k$ when $n$ is sufficiently large, we conclude that
    \[\lim_{n\to\infty} \E_{B_n}\left[\left(\frac{X_n-n/2}{\sqrt{(n+1)/12}}\right)^k\right] = \lim_{n\to\infty} \E_{\lambda_n,\mu_n}\left[\left(\frac{X_n-n/2}{\sqrt{(n+1)/12}}\right)^k\right].\]
    The result now follows from the Method of Moments and Proposition~\ref{prop:norm-allBn}.
\end{proof}

\section{Conclusion}\label{section:conclusion}

In this paper, we have introduced a notion of partial colored permutations, and we used this to study the moments of statistics $X:\mathfrak{S}_{n,r}\to \mathbb{R}$ on conjugacy classes $C_{\rpart}$ of $\mathfrak{S}_{n,r}$ without short cycles. In the case of $\mathfrak{S}_n$, our results reduce to ones in \cite{GRWCPermutationStatistics}. Prior to the final version of our definitions, we considered other potential definitions and obtained variations of Theorem~\ref{thm:MainIndependence} that were weaker. The current version properly generalizes the results in \cite{GRWCPermutationStatistics}, and given the numerous connections between \cite{GRWCPermutationStatistics} and \cite{hamaker2022characters}, one natural problem is the following. 

\begin{problem}
    Use the representation theory of $B_n$ and $\mathfrak{S}_{n,r}$ to establish analogues of the results in \cite{hamaker2022characters} for colored permutation statistics with degree at most $m$.
\end{problem}

In the case of the hyperoctahedral groups $B_n$, we also derived explicit formulas for the means of various statistics, and interestingly, the $k$th moment of a statistic with degree at most $m$ may still be independent of the number of cycles of length $1,2,\ldots,mk$. The statistics on $B_n$ considered in this paper are of degree $1$ or $2$, and we found that their means on $C_{\lambda,\mu}$ depended only on $n$ and differences of the form $\Delta^k(\lambda,\mu)=m_1^k(\lambda)-m_1^k(\mu)$.

\begin{problem}
    Undertake a study of statistics on $B_n$ whose mean (or higher moments) on $C_{\lambda,\mu}$ depend only on $n$ and differences of the form $m_i^k(\lambda)-m_i^k(\mu)$.
\end{problem}

Since $\mathfrak{S}_n$ and $B_n$ are respectively the type $A$ and type $B$ Coxeter groups, the following is another natural problem to consider next.

\begin{problem}
    Establish analogues of the results in this paper for the type $D$ Coxeter groups.
\end{problem}

One further extension of this problem is to consider the infinite complex reflection groups $G(r,p,n)$. Each $G(r,p,n)$ is the index $p\mid r$ subgroup of $\mathfrak{S}_{n,r}$, and when $r=p=2$, these reduce to the type $D$ Coxeter groups.

\section*{Acknowledgements}

We thank Yan Zhuang for alerting us to the work of Hamaker and Rhoades \cite{hamaker2022characters}, and Zach Hamaker for helpful discussions regarding connections to \cite{hamaker2022characters}. We would also like to thank the referees for reviewing this paper and providing feedback to significantly improve it. 

An extended abstract for this paper previously appeared in the Proceedings for the 36th International Conference on Formal Power Series and Algebraic Combinatorics (FPSAC) \cite{GRWCFPSAC}.

\bibliographystyle{alphaurl}
\bibliography{Bibliography}

\end{document}